\documentclass[10pt]{siamltex}
\usepackage{hyperref}
\usepackage{amsmath}
\usepackage{amsfonts}
\usepackage{amssymb}
\usepackage{graphicx}
\usepackage{color}
\usepackage{enumerate}
\usepackage{bm}
\usepackage{bbm}
\usepackage{enumitem}
\usepackage{stmaryrd}
\usepackage{comment}
\usepackage{cancel}
\usepackage{soul}
\usepackage{setspace}

\graphicspath{{./}}

\input{./ENR_SCS.sty}

\title{A mixed $\ell_1$ regularization approach for sparse simultaneous approximation of parameterized PDEs}
\author{
Nick~Dexter\thanks{Department of Mathematics, University of Tennessee, Knoxville, TN 37996. email: \texttt{ndexter@utk.edu}.} 
\and Hoang~Tran\thanks{Department of Computational and Applied Mathematics, Oak Ridge National Laboratory, Oak Ridge TN 37831-6164. email: \texttt{tranha@ornl.gov}. }
\and Clayton~Webster\thanks{Department of Mathematics, University of Tennessee, Knoxville, TN 37996 and Department of Computational and Applied Mathematics, Oak Ridge National Laboratory, Oak Ridge TN 37831-6164.  email: \texttt{cwebst13@utk.edu}, \texttt{webstercg@ornl.gov}.}
}

\begin{document}

\maketitle
\begin{abstract}
We present and analyze a novel sparse polynomial technique for the simultaneous approximation of 
parameterized partial differential equations (PDEs) with deterministic and stochastic inputs.  Our approach treats the numerical solution as a jointly sparse reconstruction problem through the reformulation of the standard basis pursuit denoising, where the set of jointly sparse vectors is infinite.  To achieve global reconstruction of sparse solutions to parameterized elliptic PDEs over both physical and parametric domains, we combine the standard measurement scheme developed for compressed sensing in the context of bounded orthonormal systems with a novel mixed-norm based $\ell_1$ regularization method that exploits both energy and sparsity.
In addition, we are able to prove that, with minimal sample complexity, error estimates comparable to the best $s$-term and quasi-optimal approximations are achievable, while requiring only {\em a priori} bounds on polynomial truncation error with respect to the energy norm.
Finally, we perform extensive numerical experiments on several high-dimensional parameterized elliptic PDE models to demonstrate the superior recovery properties of the proposed approach.
\end{abstract}

\section{Introduction}
\label{sec:introduction}
This paper is concerned with the simultaneous numerical solution of a family of partial differential equations (PDEs) on a domain $D\subset\mathbb{R}^n,\, n\in \{1,\,2,\,3\}$, which arises in a magnitude of applications involving deterministic and stochastic parameters. In particular, let $\cU  = \prod_{j=1}^d \cU_j  \subset \mathbb{R}^d$ be a tensor product parametric domain and ${\cal D}$ be a differential operator with respect to $x\in D$. We consider the following parameterized boundary value problem:  
for all ${\bm y}=(y_1,\ldots, y_d)\in\cU$, find $u(\cdot, {\bm y}):{\overline{D}}\rightarrow\mathbb{R}$, 
such that 
\begin{equation}
\label{eq:genPDE}
\cD(u,\by)=0,  \quad\mbox{in } D,
\end{equation}
subject to suitable boundary conditions. The solution map ${\bm y}\mapsto u(\cdot, {\bm y})$, defined from $\cU$ into the solution space $\cV$, typically a Sobolev space 
(e.g., $H^1_0(D)$), is now well-known to be smooth for a wide class of parameterized PDEs, see, e.g., 
\cite{CDS11,HoangSchwab14,HoangSchwab12,Gunzburger:2014hi,TWZ17}. 
In such cases, global polynomial approximation that exploits the smoothness of the solutions with respect to the parametric vector 
${\bm y}\in\cU$, are appealing approaches for solving \eqref{eq:genPDE}.  Such techniques utilize a linear combination of suitable deterministic polynomial basis functions and often feature fast convergence with reasonable cost. 
Examples of common truncated global polynomial approximation methods include Taylor expansions of multivariate monomials 
\cite{CDS11,CCDS13}, projections onto an $L^2(\cU, \varrho)$ orthonormal basis with respect to a weight $\varrho\in L^{\infty}(\mathbb{R}_+)$, 
as well as collocating at a set of nodes associated with tensorized  Lagrange polynomials 
\cite{BNT07,NTW08,Webster:2007tq}.

However, the accuracy of all these approximations heavily depend on the choice of polynomial subspace used in their construction. A na\"{i}ve selection of such a subspace can lead to an inefficient approximation scheme, while finding a good choice can often be a challenging computational task, especially when the dimension of the parameter domain is high, as detailed in \cite{MR2421041,CCDS13,TWZ17}. 
Polynomial approximation via compressed sensing (CS), see, e.g.,\cite{RauWard12,YGX12,RW15,Adcock15,Adcock15b,Adcock:2018gq,ChkifaDexterTranWebster18,JakemanNarayanZhou16}, is a new and promising approach to circumvent this problem.
Under the condition that the parameterized solution is compressible, i.e., well-represented by a sparse expansion in a given orthonormal system, these methods are able to reconstruct the largest terms of such expansions from underdetermined systems through convex optimization, iterative thresholding, or greedy selection approaches.
The key feature of CS approximations is their ability to exploit sparsity, allowing approximation on large, possibly far from optimal polynomial subspaces, with far fewer samples than the subspace dimension.
As the sparsity assumption is pertinent to many parameterized systems, see, e.g., \cite{CDS11,CCS14,BCDM16}, it is no surprise that 
CS-based polynomial approximation has attracted growing interest in the area of computational high-dimensional PDEs in recent years \cite{DO11, MG12,  YK13, Rauhut2017, PHD14, HD15, PengHampDoos15}.

Yet, there is still a significant gap between current CS-based techniques and the recovery of fully discrete approximations to solutions of parameterized PDE problems. As traditionally developed in the field of signal and image processing \cite{CRT06,Donoho06}, CS has primarily been concerned with the reconstruction of real and complex sparse vectors. The solutions of parameterized PDEs, on the other hand, are often elements of function spaces $\cV$, e.g., Hilbert spaces. The object to be recovered in this setting, thus, is a ``signal'' whose coordinates are, e.g., Hilbert-valued (referred to as \textit{Hilbert-valued vector} or \textit{Hilbert-valued signal} herein). 
This mismatch raises the following critical issue: standard signal recovery approaches, when applied in parameterized PDE contexts, do not allow direct approximation of the entire solution map $\by \mapsto u(\cdot, \by) \in \cV$, solving \eqref{eq:genPDE} in precise terms.
Instead, standard CS-based approaches are only capable of approximating functionals of the solution, i.e., maps of the form $\by \mapsto G(u(\by)) \in \mathbb{R}$, where $G: \cV \to \mathbb{R}$ is a functional of $u$. In particular, many of the existing works (cited above) perform \textit{point-wise reconstruction} of the solution, i.e., reconstruction of the map $\by \mapsto G_{x^\star}(u(\by))$ with $G_{x^\star}(u(\by)) = u(x^{\star},\by) $ for a fixed $x^{\star} \in D$.

Fully discrete {\em point-wise compressed sensing} (PCS) approximations can be obtained from the point-wise reconstructions, provided the reconstructions are performed at all prescribed abscissas in the physical discretization, using, e.g., least squares regression or piecewise polynomial interpolation.
Nonetheless, we find that this practice has some distinct limitations. 
First, the decay of the polynomial coefficients may vary significantly over points in $D$, leading to different levels of accuracy among pointwise reconstruction and, consequently, less efficient global evaluation. Secondly, {\em a priori} estimates of the tail expansion are required at every selected node, and the combined cost over all nodes can be hefty. On the other hand, if, instead, an inexpensive, conservative estimate is used as a uniform bound of all relevant point-wise tail expansions, much less accurate approximations are expected.

In this paper, we present and analyze a novel sparse approximation technique for {\em global reconstruction} of solutions to parameterized PDEs, i.e., approximation over the entire physical domain $D$ of the map $\by\mapsto u(\by)$ 
solving \eqref{eq:genPDE}. Our approach, denoted {\em simultaneous compressed sensing} (SCS), is based on a reformulation of the standard {\em basis pursuit denoising} (BPDN) constrained convex optimization problem. 
The key difference is in the choice to regularize with respect to a mixed, sparsity-inducing norm involving the energy norm associated with the space $\cV$.
We establish the theoretical support of this strategy via several extensions of compressed sensing concepts such as the {\em restricted isometry property} (RIP) and {\em null space property} (NSP) for real and complex sparse vectors to the Hilbert-valued setting, under which sample complexity and convergence estimates are acquired.
Our results show that the sample complexity requirements enjoyed by compressed sensing extend naturally to the problem of simultaneous compressed sensing, implying that the number of samples needed to reconstruct to a desired sparsity level $s$ scales only logarithmically with the size of polynomial subspace.
Moreover, extension of standard error estimates for basis pursuit denoising to the SCS setting shows that approximations obtained with our approach achieve errors comparable to the best $s$-term approximations and expansion tail in energy norms, which are often the norms of interest in parameterized PDE problems, \cite{CDS11}. 
It seems non-trivial to obtain such approximation results with the standard point-wise CS-based  reconstruction approach. 
 
However, this approach requires solving an unconventional $\ell_1$-minimization problem involving Hilbert-valued signals, where the mixed $(\mathcal{V}, \ell_1)$ regularization term is the sum of energy norms of the vector coordinates. To address this challenge, we derive and implement an iterative minimization procedure, based on an extension of the {\em fixed point continuation} (FPC) \cite{HYZ08} and {\em Bregman iterative regularization} \cite{YOGD08} algorithms to Hilbert-valued setting. 
The proposed approach is also compatible with many popular methods of physical discretization, including, e.g., the finite element, finite difference, and finite volume methods.

We also compare the performance of our SCS technique with the PCS approach when both approximations are constructed with an identical finite element discretization of the physical domain. 
We observe that, while they often feature similar computational costs, the coupled approach always displays remarkably better accuracy, thus should be the method of choice in global reconstruction. These results also verify the effectiveness of our proposed mixed $(\mathcal{V}, \ell_1)$-minimization solver. 
The accomplished results are positive, demonstrating that SCS is particularly attractive in highly anisotropic high-dimensional parameterized PDE problems. 
As we shall see through extensive numerical experiments, SCS is able to correctly identify the most important terms of sparse polynomial expansions of the solution to \eqref{eq:genPDE}, without an adaptive selection procedure. For some highly anisotropic parametric PDEs, our results show that SCS produces approximations which are comparable or superior to several others examined, including standard stochastic Galerkin and stochastic collocation methods.

Finally, the outline of the paper is as follows. In Section \ref{sec:setting} we introduce the mathematical problem and the main notation used throughout. {While the current work focuses on sparse approximation of solutions to parameterized elliptic PDE models, we remark that the developments herein may be readily applied to any parameterized systems satisfying a sparsity assumption.}
In Section \ref{sec:methodology}, we describe the main contribution of this work; our novel sparse Hilbert-valued vector reconstruction approach, including the reformulation of the BPDN problem necessary to achieve our results.
We also describe similarities between SCS and the problem setting of joint-sparsity, noting the differences between both problems when combined with finite element approximations.
In Section \ref{sec:Framework}, we establish the theoretical foundation for the global reconstruction of Hilbert-valued signals through SCS. There, sample complexity guarantees and convergence estimates are proved through simple extensions of results and concepts from compressed sensing. Section \ref{sec:algorithms} is devoted to describing a fast iterative algorithm for finding sparse approximations of parameterized PDE solutions. 
Finally, numerical results and comparisons illustrating the theoretical results and demonstrating the efficiency of our method are given in Section \ref{sec:Hilbert_valued_numerical_results}.   

\section{Background on parameterized PDE problems}
\label{sec:setting}
As discussed in the previous section, the sparse approximation methods discussed in this work are relevant to solutions of parameterized PDEs obeying a certain sparsity assumption. 
Such sparsity results have been shown for a variety of parameterized PDE problems under additional assumptions on the parametric dependence.
In this section, we consider the simultaneous solution of a parameterized linear elliptic PDE, having sufficient parametric regularity to establish the desired sparsity results. 
In particular, for all $\by\in\cU$, find $u(\cdot,\by): D \to \R$ such that 
\begin{align}
\label{eq:model_problem}
\left\{\begin{array}{rll} 
-\nabla \cdot \left( a(x,\by) \nabla u(x,\by) \right) \hspace{-0.25cm} &= f(x)  &\forall x\in D, \; \by \in \cU \\
                                                u(x,\by) \hspace{-0.25cm} &= 0  &\forall x\in \partial D, \; \by \in \cU,
\end{array}\right.
\end{align}
where $f\in L^2(D)$ is a fixed function of $x$, and $D\subset \R^n$, $n\in\{1,2,3\}$, is a bounded Lipschitz domain.
We will often suppress the dependence on $x\in D$, writing $a(\by) = a(\cdot,\by)$ and similarly $u(\by) = u(\cdot,\by)$.
Problem \eqref{eq:model_problem} is relevant to both {\em deterministic} and {\em stochastic} modeling contexts, see, e.g., \cite{Gunzburger:2014hi} for more details.

We focus on problem \eqref{eq:model_problem} under the following assumptions, guaranteeing well-posedness and parametric regularity of the solution $u$ (and hence convergence of global polynomial approximations to $u$), see, e.g., \cite{CDS11,TWZ17}:
\begin{enumerate}[label=(A\arabic*)]
\item \label{ass:A1} {\em There exist constants $0<a_{\min} \leq a_{\max}<\infty$ such that $a_{\min} \leq a \leq a_{\max}$ uniformly in $\overline{D}\times \cU$;} and 
\item \label{ass:A2} {\em The complex continuation of $a$, represented as the map $a^*:\C^d \to L^\infty(D)$, is an $L^\infty(D)$-valued holomorphic function on $\C^d$.}
\end{enumerate}

\noindent The holomorphic dependence of $\by$ on the coefficient $a$ holds for many common examples of parametric dependence, including polynomial, exponential, and trigonometric functions of the variables $y_1,\ldots,y_d$. 
Under Assumption \ref{ass:A2}, one can show that the solution map $\bz \mapsto u(\bz)$ from \eqref{eq:model_problem} is holomorphic in an open neighborhood of the domain $\cU$, see \cite{CDS11,TWZ17} for more details.

Finally, in this setting we set $\cV = H_0^1(D)$, the space of square integrable functions of $x\in D$, having zero trace on the boundary $\partial D$, and square integrable distributional derivatives. We also 
use the notation $L^2_\varrho(\cU;\cV)$ to denote the weighted Bochner space of mappings $\by\mapsto u(\cdot,\by)\in\cV$. 
The {\em parametric weak form} of problem \eqref{eq:model_problem} is given by: {\em find $u\in \Htworho$ such that $\forall v\in \Htworho$} 
\begin{equation}
\begin{aligned}
\label{eq:weak_problem}
\int_\cU \mathcal{B}[u,v](\by) \varrho(\by)\; d\by = \int_\cU F(v) \varrho(\by)\; d\by,
\end{aligned}
\end{equation}
where
$
\mathcal{B}[u,v](\by) = \int_D a(x,\by) \nabla u(x,\by) \cdot \nabla v(x,\by) \;dx, \text{ and } F(v) = \int_D f(x) v(x,\by) \;dx.
$
For convenience, we will often use the abbreviation $\mathcal{B}(\by) = \mathcal{B}[\cdot,\cdot](\by)$. 
Assumption \ref{ass:A1} and the Lax-Milgram lemma ensure the existence and uniqueness of the solution $u$ to \eqref{eq:weak_problem} in $\Htworho$.

\section{Problem setting and methodology}
\label{sec:methodology}
The main contribution of this work, discussed in this section, is an extension of standard compressed sensing theory to the recovery of sparse vectors $\bc$ whose coefficients $(\bc_\bmnu)_{\bmnu\in\cJ}$, 
for any finite subset 
$\cJ \subset \cF:=\mathbb{N}_0^d $,
belong to a general Hilbert space $\cV$, i.e., in the context of \eqref{eq:model_problem} are functions of $x\in D$.
Through this extension, we are able to derive problems, and the resulting solution algorithms, which enable sparse recovery of Hilbert-valued signals. We call our approach {\em simultaneous compressed sensing} (SCS), which can be viewed as a generalization of the well-known challenge of joint-sparse recovery, see, e.g., \cite{Cotter2005,Gribonval2008,MishaliEldar08,EldarRauhut10}.
The key to our approach is a reformulation of the standard {\em basis pursuit denoising} (BPDN) problem, see, e.g., \cite[Section 4.3]{FouRau13}, in terms of a mixed norm, denoted $(\mathcal{V}, \ell_1)$, involving the $\ell_1$ sum of the energy norms of components. Our approach enables simultaneous reconstruction of solutions to parameterized PDEs over $D\times \cU$, and can be viewed as an energy norm-based method of sparse regularization.
Later in this section we will also provide a detailed discussion on the combination of our approach with finite element approximations.

In what follows we present our approach in the context of problem \eqref{eq:model_problem} as follows. 
We begin by expanding the solution $u(x,\by)$ of \eqref{eq:model_problem} in 
an ${L_{\varrho}^2({\mathcal U})}$-orthonormal basis  $(\Psi_\bmnu)_{\bmnu\in\mathcal{F}}$ according to
\be
\label{eq:expansionPC}
u (x,{\bm y}) = \sum_{\bmnu \in \mathcal{F}} \bc_\bmnu(x) \Psi_\bmnu ({\bm y}),
\ee
where $\Psi_\bmnu = \prod_{j=1}^d \Psi_{\nu_j}$ are tensor products of  ${L^2_{\varrho_j}({\mathcal U_j})}$-orthonormal polynomials, and the coefficients $\bc_\bmnu$ belong to the space $\cV$. The series 
\eqref{eq:expansionPC} is sometimes referred to as the {\em generalized polynomial chaos} 
(GPC) expansion of $u$ (see, e.g., \cite{GS02,XK02,Wiener_38}), whose 
convergence rates are well-understood. We are particularly interested in the cases that $\varrho$ is the uniform density, and for the sake of clarity we make use of the corresponding orthonormal Legendre basis of $L_{\varrho}^2(\mathcal{U})$ as $(L_\bmnu)_{\bmnu\in \cF}$. 

For any finite subset $\cJ \subset \cF$, we define the $\cV$-valued polynomial subspace 
\begin{equation}
\label{eq:V_J}
    \cV_{\cJ} := \left\{\sum_{\bmnu\in \cJ} \hat{\bc}_\bmnu(x) \Psi_\bmnu(\by)\, :\, \hat{\bc}_\bmnu \in \cV \right\},
\end{equation}
and, with an abuse of notation, often use $\cV^N$ to denote a particular $\cV_\cJ$ with cardinality $N$, i.e., 
$N = \#(\cJ)$.
The {Galerkin} projection of $u$ onto $\cV_{\cJ}$ is given by
\begin{align}
\label{eq:Galerkin_approx}
u_{\cJ} (x,{\bm y}) = \sum_{\bmnu \in \cJ} \bc_\bmnu(x) \Psi_\bmnu ({\bm y}).
\end{align}
There are many methods for approximately resolving the Hilbert-valued vector of coefficients $\bc = (\bc_{\bmnu})_{\bmnu\in\cJ}\in \cV^N$, for more details see, e.g., \cite{Gunzburger:2014hi,SpectralUQ}.
An efficient evaluation of $u_{\cJ}$ would require the index set $\cJ$ to enclose all effective multi-indices (i.e., corresponding to the largest $\|\bc_\bmnu\|_{\cV}$) so that the tail error
\begin{align}
\label{eq:Galerkin_residual}
\|u - u_{\cJ} \|_{L_\varrho^2(\cU;\cV)}^2 = \| u_{\mcJ^c} \|_{L_\varrho^2(\cU;\cV)}^2 = \left\| \sum_{\bmnu \not \in \cJ} \bc_\bmnu \Psi_\bmnu \right\|_{L_\varrho^2(\cU;\mcV)}^2 = \sum_{\bmnu \not \in \mcJ} \|\bc_\bmnu\|_{\mcV}^2
\end{align}
is minimized. 
This condition can be easily fulfilled if we are able to select $\cJ$ to be a very large index set. 
However, for traditional approximation approaches, the size of $\cJ$ is often constrained by the computational budget; as the cost of computing grows at least linearly in cardinality $N$, while $N$ often grows exponentially in the dimension $d$ for many choices of $\cJ$, see, e.g., \cite{BNTT_comp,CCH14,CDL13,Dexter2016,Gunzburger:2014hi}. This consideration of cost is described in the following two remarks and motivates the use of adaptive strategies which seek to construct $\cJ$ in an optimal or near optimal manner.

\begin{remark}[Best $s$-term approximations]
\label{rem:best_s_term_approx}
Given oracle knowledge of the sequence $(\bc_\bmnu)_{\bmnu\in\cF}$, the optimal choice, denoted the {\em best $s$-term approximation}, constructs $\cJ_s$ corresponding to the $s$ most important terms by minimizing \eqref{eq:Galerkin_residual}.
The seminal works \cite{CDS10,CDS11} established algebraic, dimension-independent, rates of convergence for such best $s$-term Legendre and Taylor approximations of solutions to elliptic PDEs, having affine dependence on infinitely many variables $y_i$. 
Subsequent studies have extended best $s$-term estimates to parameterized parabolic and hyperbolic PDEs, non-affine parameterizations, and also Chebyshev approximations of such systems, see, e.g., \cite{BAS09,CCDS13,CCS14,HS13,HS13b,HoangSchwab12,HoangSchwab14,HS13c,Todor2007}. 
The basic strategy of analyzing best $s$-term approximations is to establish the $\ell^p$-summability, for the smallest $p\in (0,1)$, of the sequence $(\|\bc_\bmnu\|_\cV)_{\bmnu\in\cF}$.
The Stechkin inequality (see, e.g., \cite{DeV98}), can then be applied to obtain a rate of the form
\begin{align}
\label{eq:best_s_term_error}
\| u - u_{\cJ_s} \|_{L^2_\varrho(\cU;\cV)} \le \left\| ( \|\bc_\bmnu\|_\cV )_{\bmnu\in\cF} \right\|_{\ell^p(\cF)} s^{1/2-1/p},  \qquad \forall s \in \N,
\end{align}
where $u_{\cJ_s}$ is the best $s$-term approximation to $u$.
In general, explicit estimates of $\|(\|\bc_\bmnu\|_\cV)_{\bmnu\in\cF} \|_{\ell^p(\cF)}$ are unavailable without access to $(\bc_\bmnu)_{\bmnu\in\cF}$.
Moreover, \eqref{eq:best_s_term_error} often holds for a continuum of values of $p$, with stronger rates (smaller $p$) coupled with larger coefficients.
Nonetheless, estimates of this form provide a useful theoretical benchmark of performance for sparse approximations in the infinite-dimensional setting. 
\end{remark}

\begin{remark}[Quasi-optimal approximations]
\label{rem:quasi_optimal_approx}
On the other hand, given sharp bounds on the coefficients $\|\bc_\bmnu\|_\cV$, one can construct the set $\cJ_s$ corresponding to the $s$ largest of these bounds. 
The resulting approximations, denoted {\em quasi-optimal approximations}, see, e.g., \cite{BNTT14,BTNT12,TWZ17},  are relevant to a wide class of PDEs having finite-dimensional parametric dependence.
Consider a multi-indexed sequence of coefficient estimates on $(\|\bc_\bmnu\|_\cV)_{\bmnu\in\cF}$ written in the form $(e^{-b(\bmnu)})_{\bmnu\in\cF}$, i.e., satisfying
\begin{align}
\label{eq:coeff_bounds}
\|\bc_\bmnu\|_{\cV} \lesssim e^{- b (\bmnu)} \qquad \forall \bmnu\in\cF.
\end{align} 
When $b$ satisfies \cite[Assumption 3]{TWZ17}, 
and $\cJ_s$ is constructed from the $s$ largest bounds of the sequence $(e^{-b(\bmnu)})_{\bmnu\in\cF}$, one can apply \cite[Theorem 2]{TWZ17} to show that for any $\varepsilon>0$ there exists an $s_\varepsilon>0$ such that 
\begin{align}
\label{eq:quasi_optimal_rate}
\left\| u - \sum_{\bmnu \in \cJ_s} \bc_\bmnu \Psi_\bmnu \right\|_{L^2_\varrho(\cU;\cV)} \le C_\varepsilon \sqrt{s} \exp \left[ - \left( \frac{\kappa \, s}{(1 + \varepsilon)} \right)^{1/d} \right],
\end{align}
for every $s> s_\varepsilon$, where $C_\varepsilon >0$ is a very mild constant independent of $s$.
Here $\kappa$ corresponds to the volume of a limiting polytope $\mathcal{Q}$ bounding the sequence $(e^{-b(\bmnu)})_{\bmnu\in\cJ_s^c}$. Optimal values of $\kappa$ can be shown, given explicit forms of the function $b(\bmnu)$, see, e.g., the results of \cite[Propositions 4 \& 5]{TWZ17} where $\kappa$ is derived in the cases of the Taylor and Legendre systems.
This estimate is asymptotically sharp as $s\to \infty$, providing the optimal rate in the finite-dimensional setting, and carries an explicit dependence on the dimension. 
\end{remark}

Compressed sensing represents a new and promising strategy to approximate sparse polynomial expansions without any complicated a priori subspace selection. This approach is advantageous in that the growth of sample complexity with respect to the size of $\cJ$ is mild (logarithmic), thus allowing the approximation of $u$ on a large, possibly far from optimal index set. The sparsity pattern of the polynomial expansion can then be detected and the large coefficients can be accurately recovered via convex optimization procedures.

The extension of the compressed sensing framework to the global reconstruction of parameterized solutions $u$ to \eqref{eq:model_problem} is fairly straightforward.
For any $N\in \mathbb{N}_0$, $p\ge 1$, and $\bz,\bz' \in \cV^{N}$, define the norm and inner product
\begin{align}
\label{eq:norm_and_inner_product_VN}
\|{\bz}\|_{\cV,p} := 
\left(\sum_{\bmnu\in\cJ} \| \bz_\bmnu\|_{\cV}^p\right)^{1/p},
\quad
\text{and}
\quad 
\langle \bz,\bz' \rangle_{\cV,2} :=  
\sum_{\bmnu\in\cJ} \langle \bz_\bmnu , \bz'_\bmnu \rangle_\cV.
\end{align} 
In the case of $\cV = H_0^1(D)$, certainly $\langle \bz_\bmnu, \bz_\bmnu'\rangle_\cV :=  \int_D \nabla \bz_\bmnu \cdot \nabla \bz'_\bmnu \; {d}x$ and $\|\bz_\bmnu\|_{\cV}=\sqrt{\langle\bz_\bmnu,\bz_\bmnu\rangle_\cV}$.
Let $\supp(\bz) := \{\bmnu \in \cJ : \bz_\bmnu \not\equiv \bm{0}\}$ denote the support of $\bz\in\cV^N$ and $\|{\bz}\|_{\cV,0} := \#(\supp(\bz))$.
Given $p >0$, we define the $\|\cdot\|_{\cV,p}$-error of the best $s$-term approximation to $\bz\in\cV^N$ as
\begin{equation}
\sigma_s(\bz)_{\cV,p} = \inf_{\hat\bz\in\cV^N,\|\hat\bz\|_{\cV,0}\le s} \|\bz - \hat\bz\|_{\cV,p}.
\end{equation}
We say that $\bz\in\cV^N$ is $s$-sparse when $\|\bz\|_{\cV,0} \leq s$, and compressible when $\sigma_s(\bz)_{\cV,p} \to 0$ quickly in $s$.

The measurement scheme for SCS recovery borrows directly from compressed sensing-based polynomial approximation schemes, which we repeat as follows. 
Given an index set $\cJ$, with $N=\#(\cJ)$ large enough to ensure the tail $u_{\cJ^c}$ is negligible, for some $m\ll N$, we generate $m$ i.i.d. random samples $({\bm y}_k)_{k=1}^m$ from the measure $\varrho$ in the parametric domain $\mathcal{U}$.
The goal is then to find an approximation $u_\cJ^{\#}$ of $u_\cJ$ from \eqref{eq:Galerkin_approx} of the form
\begin{align}
\label{eq:approx_of_Galerkin}
 u^{\#}_\cJ(x,\by) = \sum_{\bmnu\in \cJ} \bc^{\#}_\bmnu(x) \Psi_\bmnu(\by),
\end{align}
which is comparable to best $s$-term approximation to $u$.
Let $\bA \in \R^{m\times N}$ be the normalized sampling matrix containing the samples of the basis $(\Psi_{\bmnu})_{\bmnu\in\cJ}$, $\bu$ the normalized observations of the parameterized solution at the points $\by_i$, and $\bm{e}$ the normalized expansion tail, given by
\begin{equation}
\label{eq:defA}
\bA := \left(\frac{\Psi_\bmnu(\by_i)}{\sqrt{m}}\right)_{\substack{1\leq i\leq m\\ \bmnu \in \cJ}}, \quad 
\bu := \left(\frac{u(\by_i)}{\sqrt{m}}\right)_{1\le i \le m}, \quad\text{and}\quad 
\bm{e} := \left(\frac{u_{\cJ^c}(\by_i)}{\sqrt{m}}\right)_{1\le i \le m}
\end{equation}
respectively.
Then $\bu,\bm{e}\in\cV^m$, and the vector $\bc=(\bc_\bmnu)_{\bmnu\in\cJ} \in \mathcal{V}^N$ of exact coefficients satisfy $\bA \bc + \bm{e} = \bu $. Assuming an upper estimate of the tail $\|\bm{e}\|_{\cV,2} \le \eta /\sqrt{m}$ is known, one has $\|\bA \bc - \bu\|_{\cV,2} \le \eta /\sqrt{m}$, which motivates us to find the approximation $\bc^{\#} = (\bc_\bmnu^{\#})_{\bmnu \in \cJ}$ among all $\bz\in \cV^N$ such that $\|\bA \bz - \bu\|_{\cV,2} \le \eta /\sqrt{m}$. 
Having the $\ell_1$ minimization approach for real and complex sparse recovery in mind, it is natural to consider the following modification of the BPDN problem as:
\begin{align}
\label{eq:SCS_BPDN}
\text{minimize}_{\bz\in \cV^N} \; \|\bz\|_{\cV,1}\quad\text{ subject to }\quad\|\bA \bz - \bu\|_{\cV,2} \le \eta /\sqrt{m}.
\end{align}  
Here, the Hilbert-valued vector is minimized with respect to an $\ell_1$-norm defined as the sum of the magnitude of coordinates in the energy norm. 
We will sometimes refer to problem \eqref{eq:SCS_BPDN} as the SCS-BPDN problem.

\begin{remark}[Relaxation of the constrained minimization problem \eqref{eq:SCS_BPDN}]
\label{rem:SCS_l_V1}
We note that the constrained SCS-BPDN problem \eqref{eq:SCS_BPDN} can be solved by the related unconstrained convex minimization problem
\begin{align}
\label{eq:SCS_l_V1}
\textnormal{minimize}_{\bz \in \cV^N} \|\bz\|_{\cV,1} + \frac{{\mu}}{2} \|\bA \bz - \bu \|_{\cV,2}^2.
\end{align}
This problem provides solutions to problem \eqref{eq:SCS_BPDN} when the penalty parameter ${\mu}$ is chosen accordingly to the bound on the residual.
Many existing algorithms for single sparse recovery through standard BPDN can be modified in a straightforward manner to obtain solutions to \eqref{eq:SCS_BPDN} and  \eqref{eq:SCS_l_V1}, e.g., 
forward-backward splitting \cite{DexterTranWebster2018,HYZ08,HYZ10} and Bregman iterations \cite{OBGXY05}, as will be discussed in Section \ref{sec:algorithms}.
\end{remark}

\begin{remark}[Estimation of truncation errors]
\label{rem:tail_est}
In general, an estimate ${\eta}$ of the expansion tail is required in order to obtain accurate solutions to \eqref{eq:SCS_BPDN}. As aforementioned, we assume $\eta$ to be an upper bound of $\(\sum_{i=1}^m \|u_{\cJ^c}(\by_i)\|^2_{\cV}\)^{1/2}$ for theoretical uniform recovery. Note that
$$
\mathbb{E}\(\frac 1m\sum_{i=1}^m \|u_{\cJ^c}(\by_i)\|^2_{\cV}\) = \left\| u - u_\cJ \right\|_{L^2_\varrho(\cU;\cV)}^2 ,  
$$
$\eta$ is often considered interchangeably as an adequate estimate of $\sqrt{m}\left\| u - u_\cJ \right\|_{L^2_\varrho(\cU;\cV)}$, particularly in numerical tests. Representing 
$
\dfrac{\eta}{\sqrt{m}} = C_{\eta} \left\| u - u_\cJ \right\|_{L^2_\varrho(\cU;\cV)}, 
$ we can reasonably assume that 
$
C_{\eta}
$ 
is $\mathcal{O}(1)$. We remark that it is possible to acquire some uniform recovery when truncation errors are unknown a priori, see \cite{BrugiapagliaAdcock18,ABB17}.  
\end{remark}

\begin{remark}[SCS in the context of finite element discretizations]
Our previous work \cite{DexterTranWebster2018} described the connection between infinite dimensional joint-sparse recovery and SCS for solving \eqref{eq:genPDE} through problem \eqref{eq:SCS_l_V1}.
Indeed, these problems are equivalent in the sense that solutions $\bc$ obtained from SCS through \eqref{eq:SCS_l_V1} (given exact samples $\bu = (u(\by_i))_{i=1}^m\in \cV^m$ solving \eqref{eq:genPDE}) can be identified with matrices $\hat{\bc}$ solving the $\ell_{2,1}$-regularized joint-sparse convex minimization problem (with the same data, expanded in a countable orthonormal basis of $\cV$).

On the other hand, SCS is relevant to many popular methods of discretization, including the finite element, volume, and difference methods.
We briefly describe solution of \eqref{eq:SCS_BPDN} in the context of finite element approximations.
Let $(\mathcal{T}_h)_{h>0}$, be a family of shape regular triangulations of $D$, parameterized by maximum mesh size $h\to 0$, and let $\cV_h\subset \cV$ be the finite element space of piecewise continuous polynomials on $\mathcal{T}_h$. 
Given a basis $(\varphi_k)_{k=1}^{\Jh}$ for $\cV_h$ with $K_h = \dim(\cV_h)$, we can approximate the Hilbert-valued vector $\bc = (\bc_\bmnu)_{\bmnu\in\cJ}\in\cV^N$ by finite element expansions of its coefficients, defining $\bc^h = (\bc_\bmnu^h)_{\bmnu\in\cJ}$ as 
\begin{align}
\label{eq:c_nu_FEM_expansion}
\bc_{\bmnu}^h := \sum_{k = 1}^{K_h} \hat{c}_{\bmnu,k}^h \varphi_k \in \cV_h  \qquad \forall \bmnu\in\cJ,
\end{align}
where $\hat{c}_{\bmnu,k}^h \in \R$ for $k= 1,\ldots, \Jh$.
As in the infinite dimensional case, we can identify the coefficients of the expansion \eqref{eq:c_nu_FEM_expansion} as rows of a matrix $\hat{\bc}^h \in \R^{N \times \Jh}$. 
However, unlike in that case, we do not have $\|\hat{\bc}^h\|_{2,q} \equiv \|\bc^h\|_{\cV,q}$ for $q\ge 1$, but instead
\begin{align}
\label{eq:2q_Vq_norm_equivalence}
c_{h,n} \|\hat{\bc}^h \|_{2,q} \le \|\bc^h \|_{\cV,q} \le C_{h,n} \|\hat{\bc}^h \|_{2,q},
\end{align}
for some constants $0 < c_{h,n} \le C_{h,n}$ which may depend on $h$ and $n$. 
{Using problem \eqref{eq:model_problem} as an example, if $D\subset \R^n$ is a Lipschitz domain in $n = 2$ physical dimensions, and $(\varphi_k)_{k=1}^{\Jh}$ is a piecewise-linear finite element basis, associated with a quasi-uniform and shape-regular mesh $\mathcal{T}_h$ of $D$, then $c_{h,n}$ is $\mathcal{O}(h^{2})$ and $C_{h,n}$ is $\mathcal{O}(1)$. For more details, see, e.g., \cite{MR1278258}.}
Nonetheless, given the representation \eqref{eq:c_nu_FEM_expansion}, it is easy to compute $\|\bc_\bmnu^h \|_\cV$ in finding solutions to \eqref{eq:SCS_BPDN}. More details on the changes to algorithms required for solving \eqref{eq:SCS_BPDN} via \eqref{eq:SCS_l_V1} with finite element discretizations will be described in Section \ref{sec:algorithms}.
\end{remark}

\section{Framework for SCS recovery of solutions to parameterized PDEs} 
\label{sec:Framework}

In this section, we provide a framework for the approximation of solutions to parameterized PDE problems through the SCS approach outlined in Section \ref{sec:methodology}.
This framework relies on the theoretical coefficient bounds on the solution to the parameterized PDE problems of type \eqref{eq:genPDE}.
With these bounds and {\em a priori} error estimates shown for quasi-optimal polynomial approximations in \cite{TWZ17}, we are able to derive convergence rates for such approximations, relevant to many practical finite-dimensional parametric PDE problems.
The rest of this section is organized as follows. 
Section \ref{subsec:Hilbert_valued_uniform_recovery} discusses uniform recovery results and error estimates for solving the SCS-BPDN problem \eqref{eq:SCS_BPDN} in term of best $s$-term errors and tail bounds.
Section \ref{subsec:Hilbert_valued_error_estimates} combines the quasi-optimal error estimates with the uniform recovery for the SCS-BPDN problem, and establishes the convergence of our approach with respect to the number of samples $m$.

\subsection{Uniform recovery of Hilbert-valued signals}
\label{subsec:Hilbert_valued_uniform_recovery}
Straightforward extensions of concepts and results from compressed sensing and joint-sparse recovery can be made to ensure uniform recovery of Hilbert-valued signals via $\ell_{\cV,1}$-relaxation.
Well-known concepts such as the {\em null space property} (NSP) and {\em restricted isometry property} (RIP) have Hilbert-valued counterparts.
In this section, we present Hilbert-valued versions of the NSP and RIP to guarantee uniform recovery for the SCS-BPDN problem, i.e., in the presence of noise or sparsity defects. 
To our knowledge, these extensions have not yet been provided in the considered setting, but are necessary for establishing recovery guarantees.

In the noiseless case, i.e., $\eta=0$ in  \eqref{eq:SCS_BPDN}, let $[N]$ denote the set $\{1,\ldots,N\}$ and $\bz_S$ be the Hilbert-valued vector consisting of only the elements from $\bz$ indexed by the set $S\subseteq [N]$, i.e., $\bz_j = \bm{0}$ for $j\in S^c$.
Following simple extensions of results from \cite{Lai2011,BergFriedlander10}, it can be shown that every $s$-sparse $\bc\in\cV^N$ can be uniquely recovered from $\bu = \bA \bc$ by solving \eqref{eq:SCS_BPDN} if and only if 
\begin{align*}
\|\bz_S \|_{\cV,1} < \|\bz_{S^c} \|_{\cV,1},
\end{align*}
for all $\bz \neq 0$ with $\bz\in \ker \bA \setminus \{\bm{0}\}$, and all $S\subset [N]$ with $\#(S) = s$.
In the noisy case $\eta > 0$, best $s$-term approximations of vectors $\bc\in \cV^N$ can be guaranteed (up to a constant and noise level) by the $\ell_{\cV,2}$-robust NSP.

\begin{definition}[$\ell_{\cV,2}$-robust null space property]
\label{def:l_V2_RNSP}
The matrix $\bA\in\R^{m\times N}$ is said to satisfy the $\ell_{\cV,2}$-robust null space property of order $s$ with constants $0< \rho < 1$ and $\tau > 0$ if 
\begin{align}
\|\bz_S \|_{\cV,2} \le \frac{\rho}{\sqrt{s}} \|\bz_{S^c}\|_{\cV,1} + \tau \|\bA \bz \|_{\cV,2} \quad \forall \bz\in\cV^N, \forall S\subset[N] \;\; \textnormal{with} \;\; \#(S) \le s.
\end{align}
\end{definition}

Just as in the setting of joint-sparse recovery, replacing the vector norms by $\ell_{\cV,q}$-norms, we can obtain the following error estimate for solving the SCS-BPDN problem.

\begin{proposition}
\label{prop:lpV_convergence}
Let $1 \le q \le 2$. Suppose that the matrix $\bA\in\R^{m\times N}$ satisfies the $\ell_{\cV,2}$-robust NSP of order $s$ with constants $0< \rho < 1$ and $\tau > 0$. Then, for any $\bc\in\cV^N$, the solution $\bc^\#$ to \eqref{eq:SCS_BPDN} with $\bu = \bA\bc + \bm{e}$ and $\|\bm{e}\|_{\cV,2} \le \eta/\sqrt{m}$ approximates $\bc$ with errors given by
\begin{align}
\label{eq:SCS_BPDN_error}
\|\bc - \bc^\# \|_{\cV,q} \le \frac{C_1}{s^{1-1/q}} \sigma_{s}(\bc)_{\cV,1} + C_2 s^{1/q - 1/2} \frac{\eta}{\sqrt{m}},
\end{align}
for some constants $C_1,C_2 > 0$ depending only on $\rho$ and $\tau$. In particular, for $q=1$ and $q=2$, \eqref{eq:SCS_BPDN_error} yields
\begin{align}
\label{eq:BPDN_q=1_error_estimates}
\|\bc - \bc^\# \|_{\cV,1} & \le C_1 \sigma_s(\bc)_{\cV,1} + C_2 \eta\sqrt{\frac sm}, \\
\label{eq:BPDN_q=2_error_estimates}
\|\bc - \bc^\# \|_{\cV,2} & \le \frac{C_1}{\sqrt{s}}  \sigma_s(\bc)_{\cV,1} + C_2 \frac{\eta}{\sqrt{m}}.
\end{align}
\end{proposition}

An RIP-type condition is required to quantify the sample complexity of solving \eqref{eq:SCS_BPDN} to a prescribed accuracy.
The following result establishes the implication of the $\ell_{\cV,2}$-robust NSP from the standard RIP, relevant for SCS in tensor products of separable Hilbert spaces $\cV$.

\begin{proposition}
\label{prop:RIP_V_RIP_equivalence}
Suppose that $\cV$ is a separable Hilbert space, and that the matrix $\bA\in \R^{m\times N}$ satisfies the RIP, that is
\begin{align}
\label{eq:standard_RIP}
(1 - \delta_{2s}) \|\bz \|_2^2 \le \|\bA \bz \|_2^2 \le (1+ \delta_{2s}) \|\bz \|_2^2, \quad \forall \bz\in\R^N, \;\; \bz \;\; 2s\textnormal{-sparse},
\end{align}
with $\delta_{2s} < \frac{4}{\sqrt{41}}$. Then, $\bA$ satisfies the $\ell_{\cV,2}$-robust NSP of order $s$ with constants $0< \rho < 1$ and $\tau>0$ depending only on $\delta_{2s}$.
\end{proposition}

\begin{proof}
We show \eqref{eq:standard_RIP} is equivalent to the following $\ell_{\cV,2}$-version of the RIP, i.e., 
\begin{align}
\label{eq:V_RIP}
(1 - \delta_{2s}) \|\bz \|_{\cV,2}^2 \le \|\bA \bz \|_{\cV,2}^2 \le (1+ \delta_{2s}) \|\bz \|_{\cV,2}^2, \quad \forall \bz\in\cV^N, \;\; \bz \;\; 2s\textnormal{-sparse}.
\end{align}
The implication of the $\ell_{\cV,2}$-robust NSP in Definition \ref{def:l_V2_RNSP} then easily follows by arguments from vector recovery proofs, see, e.g., \cite[Theorem 6.13]{FouRau13}.

First, assuming \eqref{eq:standard_RIP}, let $\bz\in \cV^N$ be a $2s$-sparse Hilbert-valued vector and $(\phi_r)_{r\in\N}$ be an orthonormal basis of $\cV$.
Then for each $\bmnu\in\cJ$, $\bz_\bmnu\in\cV$ can be uniquely represented as
\begin{align*}
\bz_\bmnu = \sum_{r\in\N} \hat{\bz}_{\bmnu,r} \phi_r. 
\end{align*}
Let $\hat{\bz}\in(\ell_2(\R))^N$ be given by $\hat{\bz}_\bmnu = (\hat{\bz}_{\bmnu,r})_{r\in\N}$, $\forall \bmnu \in \cJ$ and define the $\ell_{p,q}$-norm by $\|\hat{\bz}\|_{p,q}^q = \sum_{\bmnu\in\cJ} \|\hat{\bz}_{\bmnu}\|_p^q$.
For each $r\in\N$, the vector $\hat{\bz}^{(r)} := (\hat{\bz}_{\bmnu,r})_{\bmnu\in\cJ} \in \R^N$ is $2s$-sparse, implying 
$
(1- \delta_{2s})\|\hat{\bz}^{(r)}\|_{2}^2 \le \|\bA \hat{\bz}^{(r)} \|_{2}^2 \le (1 + \delta_{2s}) \|\hat{\bz}^{(r)}\|_2^2.
$
Summing over $r\in\N$, we obtain
\begin{align*}
(1- \delta_{2s})\|\hat{\bz}\|_{2,2}^2 \le \|\bA \hat{\bz} \|_{2,2}^2 \le (1 + \delta_{2s}) \|\hat{\bz}\|_{2,2}^2.
\end{align*}
Since $\|\bz\|_{\cV,2} = \|\hat{\bz}\|_{2,2}$ and $\|\bA\bz\|_{\cV,2} = \|\bA \hat{\bz}\|_{2,2}$, see \cite[Section 1.1]{DexterTranWebster2018}, the first direction of the result follows.

On the other hand, assuming \eqref{eq:V_RIP}, let $\hat{z}$ be a $2s$-sparse vector in $\R^N$. Fix an $r\in \N$, one can easily verify that the Hilbert valued vector $\bz\in\cV^N$ defined by $\bz_\bmnu := \hat{z}_\bmnu \phi_r$ $\forall \bmnu\in\cJ$ satisfies $\|\bz \|_{\cV,2} = \|\hat{z}\|_2$ and $\|\bA\bz\|_{\cV,2} = \|\bA\hat{z}\|_2$. Since $\bz$ is $2s$-sparse, the result follows.
\end{proof}

Proposition \ref{prop:RIP_V_RIP_equivalence} implies that sample complexity results shown for solving the single-sparse BPDN problem
hold for the SCS-BPDN problem \eqref{eq:SCS_BPDN} as well. 
Therefore, following our previous work \cite{ChkifaDexterTranWebster18}, with 
$\Theta = \sup_{\bmnu\in\cJ} \|\Psi_\bmnu\|_{L^\infty(\cU)}$, 
given $m\asymp \Theta^2 s\log^2(s)\log(N)$, SCS reconstruction through problem \eqref{eq:SCS_BPDN} achieves an error of $\|\bc - \bc^\# \|_{\cV,2} \lesssim \sigma_s(\bc)_{\cV,1}/\sqrt{s} + \eta/\sqrt{m}$
 with probability $1-N^{-\log(s)}$.

\subsection{Error estimates for Hilbert-valued recovery}
\label{subsec:Hilbert_valued_error_estimates}

With Propositions \ref{prop:lpV_convergence} - \ref{prop:RIP_V_RIP_equivalence} and error estimates as in \cite[Theorem 2]{TWZ17} for quasi-optimal approximations, we are now able to provide convergence rates for approximations to parameterized PDEs obtained through the SCS recovery techniques of Section \ref{sec:methodology}. 
The next theorem provides a benchmark for performance of sparse Hilbert-valued reconstruction in the general setting of solutions to problem \eqref{eq:genPDE}, under the condition that the solution $u$ has parametric expansion with coefficients $(\bc_{\bmnu})_{\bmnu\in\cF}$ as in \eqref{eq:expansionPC} satisfying $\|\bc_{\bmnu}\|_\cV \lesssim e^{-b(\bmnu)}$ for every $\bmnu\in\cF$, with $b(\bmnu)$ obeying Assumption 3 from \cite{TWZ17}. For brevity, we do not detail that assumption herein, but remark that it is weak and satisfied by all existing coefficient estimates we are aware of.

\begin{theorem}
\label{thm:SCS_convergence}
Let $\varepsilon > 0$ be arbitrary and assume that the solution $u$ to \eqref{eq:genPDE} with parametric expansion \eqref{eq:expansionPC} has coefficients $(\bc_{\bmnu})_{\bmnu\in\cF}$ satisfying $\|\bc_\bmnu\|_\cV \le e^{-b(\bmnu)}$ for every $\bmnu\in\cF$ with $b:[0,\infty)^d \to \R$ also satisfying \cite[Assumption 3]{TWZ17}.
Denote by $\cJ_s$ the set of indices corresponding to the $s$ largest bounds of the sequence $(e^{-b(\bmnu)})_{\bmnu\in\cF}$ with $s$ large enough such that \eqref{eq:quasi_optimal_rate} holds.
Let $\cJ$ be such that $\cJ_s \subseteq \cJ$, and assume that the number of samples $m\in\N$ satisfies
\begin{align}
\label{eq:RIP_complexity}
m\ge C \Theta^2 s \max \{\log^2(\Theta^2 s) \log(N), \log(\Theta^2 s) \log(\log(\Theta^2 s) N^{\log(s)}) \},
\end{align}
with $\Theta = \sup_{\bmnu\in\cJ} \|\Psi_\bmnu\|_{L^\infty(\cU)}$ and $N= \#(\cJ)$. 
Then with probability $1-N^{-\log(s)}$, the solution $\bc^\#$ of 
\begin{align}
    \label{eq:SCS_BPDN_repeat}
\textnormal{minimize}_{\bz\in\mcV^N} \;\; \|\bz\|_{\mcV,1} \quad \textnormal{subject to} \quad \|\bA \bz - \bmu \|_{\mcV,2} \le \frac{\eta}{\sqrt{m}}
\end{align}
approximates $u$ with error
\begin{align}
\label{eq:SCS_convergence}
\left\| u - \sum_{\bmnu \in \cJ} \bc_\bmnu^\# \Psi_\bmnu \right\|_{L^2_\varrho(\cU;\cV)} \le \tilde{C}_\varepsilon \sqrt{s} \exp \left[ - \left( \frac{\kappa \, s}{(1 + \varepsilon)} \right)^{1/d} \right],
\end{align}
where $\kappa,\tilde{C}_\varepsilon >0$ are independent of $s$.
\end{theorem}

\begin{proof}
The requirement on $m$ in \eqref{eq:RIP_complexity} ensures the $\ell_{\cV,2}$-RIP and hence $\ell_{\cV,2}$-robust NSP holds for the normalized sampling matrix $\bA$ as a consequence of Proposition \ref{prop:RIP_V_RIP_equivalence}. Hence, in the considered setting, recovery up to the best $s$ term and truncation errors, as in estimate \eqref{eq:SCS_BPDN_error}, occurs with probability $1-N^{-\log(s)}$ when solving problem \eqref{eq:SCS_BPDN_repeat}, see \cite[Theorem 2.2]{ChkifaDexterTranWebster18}. 
All that remains is to combine the error estimate \eqref{eq:quasi_optimal_rate} for the quasi-optimal approximation $u_{\cJ_s}$ with the estimate for the Hilbert-valued BPDN problem under the $\ell_{\cV,2}$-RIP.
Let $u_\cJ = \sum_{\bmnu\in\cJ} \bc_\bmnu \Psi_\bmnu$ as in \eqref{eq:Galerkin_approx} and $u_\cJ^\# = \sum_{\bmnu\in\cJ} \bc_\bmnu^\# \Psi_\bmnu$ be the approximation obtained after solving \eqref{eq:SCS_BPDN_repeat}. Then 
\begin{align*}
\left\| u - u_\cJ^\# \right\|_{L^2_\varrho(\cU;\cV)} 
 \le \left\| u - u_\cJ \right\|_{L^2_\varrho(\cU;\cV)} + \left\| u_\cJ - u_\cJ^\# \right\|_{L^2_\varrho(\cU;\cV)}.
\end{align*}
From Remark \ref{rem:tail_est}, we have
$
\left\| u - u_\cJ \right\|_{L^2_\varrho(\cU;\cV)} 
 = \dfrac{\eta}{C_{\eta} \sqrt{m}}.
$
Also, by Parseval's identity,
\begin{align*}
\left\| u_\cJ - u_\cJ^\# \right\|_{L^2_\varrho(\cU;\cV)} 
= \left(\sum_{\bmnu\in\cJ} \left\| \bc_\bmnu - \bc_\bmnu^\# \right\|_{\cV}^2 \right)^{1/2} = \left\| \bc_\cJ - \bc_\cJ^\# \right\|_{\cV,2},
\end{align*}
where $\bc_\cJ = (\bc_\bmnu)_{\bmnu\in\cJ}$ and $\bc_\cJ^\# = (\bc_\bmnu^\#)_{\bmnu\in\cJ}$. Under \eqref{eq:RIP_complexity}, there follows
\begin{align*}
\left\| \bc_\cJ - \bc_\cJ^\# \right\|_{\cV,2} \le \frac{C_1}{\sqrt{s}} \sigma_s (\bc_\cJ)_{\cV,1} + \frac{C_2}{\sqrt{m}} \eta,
\end{align*}
where $C_1,C_2>0$ are the constants from Proposition \ref{prop:lpV_convergence}.
Grouping terms, we see that $u_{\cJ}^\#$ satisfies
\allowdisplaybreaks
\begin{align*}
& \left\| u - u_\cJ^\# \right\|_{L^2_\varrho(\cU;\cV)} \le \frac{C_1}{\sqrt{s}} \sigma_s (\bc_\cJ)_{\cV,1} + \left( \frac{1}{C_{\eta}} + {C_2} \right) \frac{\eta}{\sqrt{m}} \\
     \le\, & \frac{C_1}{\sqrt{s}} \left( \sum_{\bmnu\in\cJ\setminus\cJ_s} \|\bc_\bmnu\|_\cV \right) + \left( 1 + C_{\eta}{C_2} \right) \left( \, \sum_{\bmnu\in \cJ^c} \|\bc_\bmnu\|_\cV^2 \right)^{1/2} \\
    \le\, & \frac{C_1}{\sqrt{s}} \left( \sum_{\bmnu\in\cJ_s^c} e^{-b(\bmnu)} \right) + \left(1 + C_{\eta}{C_2} \right) \left( \sum_{\bmnu\in\cJ_s^c} e^{-2b(\bmnu)} \right)^{1/2} \\
     \le\,  & \frac{C_1}{\sqrt{s}} C_\varepsilon s \exp \left[ - \left( \frac{\kappa_1 s}{(1+\varepsilon)} \right)^{1/d} \right] 
 + \left(1+ C_{\eta} {C_2} \right) \sqrt{C_\varepsilon s} \left( \exp\left[ - \left(\frac{\kappa_2 s}{(1+\varepsilon)}\right)^{1/d} \right] \right)^{1/2},
\end{align*}
where $\kappa_1$ and $\kappa_2$ are the volumes of the limiting polytopes $\mathcal{Q}_1$ and $\mathcal{Q}_2$ associated with the sequences $(e^{-b(\bmnu)})_{\bmnu\in\cF}$ and $(e^{-2b(\bmnu)})_{\bmnu\in\cF}$ (for detailed definition, see \cite[Theorem 2]{TWZ17}). Observe $\kappa_2= 2^d \kappa_1$,
after substituting $\kappa_2$ and simplifying, we see that $u_\cJ^\#$ satisfies 
\begin{align*}
\left\| u - u_\cJ^\# \right\|_{L^2_\varrho(\cU;\cV)} 
    & \le \left( C_1C_\varepsilon + (1 + C_{\eta} {C_2})\sqrt{C_\varepsilon} \right) \sqrt{s} \exp \left[ - \left( \frac{ \kappa_1 \, s}{(1 + \varepsilon)} \right)^{1/d} \right],
\end{align*}
and therefore \eqref{eq:SCS_convergence} holds with $\tilde{C}_\varepsilon := \left( C_1C_\varepsilon + (1 + C_{\eta} {C_2})\sqrt{C_\varepsilon} \right) >0$ and $\kappa = \kappa_1$ when $m$ satisfies  \eqref{eq:RIP_complexity}. 
\end{proof}

\begin{remark}[Constants in Theorem \ref{thm:SCS_convergence} in the case of the orthonormal Legendre system and the parameterized elliptic PDE \eqref{eq:model_problem}] 
Theorem 2 from \cite{TWZ17} relates the constant $\kappa$ in the above result to the size and shape of the quasi-optimal index set $\cJ_s$, and gives the constant $C_\varepsilon = (4e+4\varepsilon e - 2) \frac{e}{e-1}$.
For many examples of $b(\bmnu)$ one can also derive explicit $\kappa$ and demonstrate optimality of the above result, see, e.g., \cite[Propositions 3-6]{TWZ17} for the results in the cases of the Taylor and Legendre series. We discuss this derivation for the orthonormal Legendre system $(L_\bmnu)_{\bmnu\in\cF}$, relevant to the sparse polynomial approximation methods of Section \ref{sec:methodology} and the numerical results in Section \ref{sec:Hilbert_valued_numerical_results}. 
Under Assumptions \ref{ass:A1} \& \ref{ass:A2}, one can show that the coefficients $(\bc_\bmnu)_{\bmnu\in\cF}$ of \eqref{eq:expansionPC} satisfy 
\begin{align}
\label{eq:Legendre_coeff_bounds}
\|\bc_\bmnu\|_{\cV} \le C\bm{\gamma}^{-\bmnu} \prod_{i=1}^d \sqrt{2\nu_i + 1} \qquad  \bmnu\in\cF,
\end{align}
where $\bm{\gamma}$ with $\gamma_{i} > 1$ for all $1\le i \le d$ is a multi-index parameterizing the largest complex polyellipse containing $\cU$ in $\C^d$ on which $a^*(\cdot,\bz)$ is uniformly elliptic, see, e.g., \cite[Proposition 2]{TWZ17}. 
Setting $\lambda_i  =\log(\gamma_i)$ $\forall 1\le i \le d$, $\kappa$ can be computed explicitly as $\kappa = d! \prod_{i=1}^d \lambda_i$, see \cite[Proposition 5]{TWZ17}, and \eqref{eq:SCS_convergence} becomes 
\begin{align}
\label{eq:SCS_Legendre_convergence}
\left\|u - \sum_{\bmnu\in\cJ} \bc_{\bmnu}^\# L_\bmnu \right\|_{L^2_\varrho(\cU;\cV)} \le  \widehat{C}_{\varepsilon} \, \sqrt{s} \, \exp\left[ - \left( \frac{d! \, s \,\prod_{i=1}^d \lambda_i}{(1+ \varepsilon)} \right)^{1/d} \right],
\end{align}
where $\widehat{C}_{\varepsilon} = \left(C_1 C_\varepsilon + (1 + C_{\eta} C_2)\sqrt{C_\varepsilon} \right) C$ and $C$ is the constant from \eqref{eq:Legendre_coeff_bounds}.
\end{remark}

\section{Algorithms for simultaneous compressed sensing}
\label{sec:algorithms}

In this section, we briefly present details of the algorithms used to solve the SCS-BPDN convex minimization problem \eqref{eq:SCS_BPDN} in the context of the numerical experiments of Section \ref{sec:Hilbert_valued_numerical_results}. 
As discussed in Remark \ref{rem:SCS_l_V1}, the constrained SCS-BPDN problem can be solved through the unconstrained $\ell_{\cV,1}$-regularized convex minimization problem \eqref{eq:SCS_l_V1} when $\mu$ is chosen appropriately. 
Therefore, we present modifications to the well-known Bregman iterative regularization \cite{Bregman1967,YOGD08} with forward-backward iterations \cite{Bauschke2010,Combettes2005,Daubechies2004}, which have been shown to be an efficient combination for solving the single-sparse version of \eqref{eq:SCS_l_V1}. 
We also apply a modification of the fixed-point continuation strategy from \cite{HYZ08} for improving the convergence of the forward-backward iterations, denoting the combined solver as the Bregman-FPC algorithm.
We describe our approach as follows.

Since the objective function of \eqref{eq:SCS_l_V1} is composed of both convex differentiable and convex non-differentiable parts, we apply the forward-backward iterations 
\begin{align}
\label{eq:SCS_G_tau}
G_\tau(\bx) & = \bx - \tau \bA^*(\bA\bx - \bu), \\
\label{eq:SCS_J_upsilon}
J_{\upsilon,\bmnu}(\bx) & = J_\upsilon(\bx_\bmnu) = \frac{\bx_{\bmnu}}{\|\bx_{\bmnu}\|_\cV} \cdot \max \{ \|\bx_{\bmnu} \|_\cV - \upsilon, 0\}, \qquad \bmnu\in\cJ, \\
\label{eq:SCS_S_tau}
\bx^{k+1} & = J_\upsilon \circ G_\tau (\bx^k), \qquad k \in \N_0,
\end{align}
starting from an initial guess $\bx^0\in\cV^N$, where $\tau>0$ is the step size and $\upsilon= \tau/\mu$. In practice, we set $\bx^0 = \tau \bA^* \bu$, which provides problem-specific information and is simple to compute. When combined with the FPC strategy, we apply the forward-backward iterations \eqref{eq:SCS_S_tau} to solve a sequence of subproblems of the form \eqref{eq:SCS_l_V1} with an increasing sequence of data fidelity parameters $\mu_0 < \mu_1 < \cdots \le \bar{\mu}$. 
The values of $\mu_\ell$ increase towards the final value $\bar{\mu}$ following growth rule $\mu_{\ell+1} = \min\{4^\ell \mu_0 , \bar{\mu}\}$, and the step sizes $\upsilon$ for soft-thresholding in \eqref{eq:SCS_J_upsilon} decrease for the $\ell$-th subproblem following $\upsilon_\ell = \tau/\mu_\ell$. 
The solver is supplied an initial value $\mu_0 = \frac{\tau}{0.99 \|\bx^0\|_{\cV,\infty}}$ to ensure soft-thresholding reduces negligible components of $G_\tau(\bx^0)$ to $\bm{0}$, leaving only those larger than $0.99 \|\bx^0\|_{\cV,\infty}$ to refine. 
For more details on these choices, see \cite{HYZ08,HYZ10}. We renormalize $\bA$ and $\bu$, setting $\bA \leftarrow \hat{\lambda}_{\max}^{-1/2}\bA$ and $\bu \leftarrow \hat{\lambda}_{\max}^{-1/2}\bu$, where $\hat{\lambda}_{\max}$ is the maximum eigenvalue of $\bA^*\bA$, noting that with this choice it is sufficient to choose $\tau \in [1,2)$ to guarantee convergence. 
In practice, we find $\tau=1$ to be a simple and conservative choice of step size, noting that larger values of $\tau$ may lead to improved convergence of the algorithm, see, e.g., the discussion of \cite[Sections 3.3 \& 4.2.1]{HYZ10}.

The residual weight parameter $\mu_\ell$ and step size parameter $\upsilon_\ell$ are updated once the $k$-th iteration satisfies
\begin{align}
\label{eq:mu_ell_update_criterion}
\frac{ \| \bx^k - \bx^{k-1} \|_{\cV,2} }{ \max\{ \| \bx^{k-1} \|_{\cV,2} , 1\}} < \sqrt{\frac{ \bar{\mu} }{ \mu_{\ell} }} x_{\textnormal{tol}}  \quad \textnormal{and} \quad \mu_\ell \| \bA^* (\bA \bx^k - \bu) \|_{\cV,\infty} - 1 < g_{\textnormal{tol}},
\end{align}
where $x_{\textnormal{tol}}, g_{\textnormal{tol}} > 0$ are user-supplied tolerance constants. As noted in \cite{HYZ10}, the first condition in \eqref{eq:mu_ell_update_criterion} ensures the last step is small relative to the previous iteration, while the second condition checks for complementarity at the current iteration. 
The authors further note that the second condition greatly improves accuracy, however $g_{\textnormal{tol}}$ should be large to ensure faster convergence, recommending the parameterization $x_{\textnormal{tol}} = 10^{-4}$ and $g_{\textnormal{tol}} =  0.2$. Through extensive numerical testing, our results suggest that $x_{\textnormal{tol}}$ is not as important when using Bregman iterations combined with FPC for SCS recovery, and we observe that updating the parameters $\mu_{\ell}$ and $\nu_{\ell}$ more often can improve convergence.
In practice, we find that the choice of 
$x_{\textnormal{tol}} = 1$ and $g_{\textnormal{tol}} = 0.1$,
works well for the problems considered in Section \ref{sec:Hilbert_valued_numerical_results}, while smaller values of $x_{\textnormal{tol}}$ require more forward-backward iterations, but do not improve overall errors. An explanation of this phenomenon may be found in the ``error-forgetting'' and ``error-cancelation'' properties of the Bregman algorithm when combined with FPC \cite{Yin2013}, though further testing is needed to verify our observations. We leave a more detailed investigation on best practices in parameterizing the combined solver for SCS recovery to a future work.

We combine the FPC strategy and forward-backward iterations with Bregman iterative regularization to further improve efficiency of our solver.
Bregman iterations also involve the solution of a sequence of subproblems of the form \eqref{eq:SCS_l_V1}, and the Bregman algorithm for SCS-BPDN \eqref{eq:SCS_BPDN} can be written
\begin{align*}
& \bu^{0} := \bm{0},\; \bz^{0} := 0, \numberthis \label{eq:V2_init} \\
& \textnormal{For } k = 0, 1, \ldots \textnormal{ do} \\
& \qquad \bu^{k+1} = \bu + (\bu^{k} - \bA \bz^{k}), \numberthis \label{eq:fk_+_1} \\
& \qquad \bz^{k+1} = \argmin_{\bz\in\mcV^N} \|\bz\|_{\cV,1} + \frac{\bar{\mu}}{2} \|\bA \bz - \bu^{k+1} \|_{\mcV,2}^2. \numberthis \label{eq:V2_zk_+_1}
\end{align*}
Step \eqref{eq:V2_zk_+_1} is solve using FPC with the parameters above, and
the Bregman iterations continue until 
\begin{align}
\label{eq:Bregman_stop_cond}
\|\bA \bz^k - \bu \|_{\cV,2} < b_{\textnormal{tol}},
\end{align}
for some $k \in \N$ and $b_{\textnormal{tol}}>0$. 
To achieve the residual constraint of problem \eqref{eq:SCS_BPDN}, we choose $b_{\textnormal{tol}}$ proportional to $\eta/\sqrt{m}$, i.e., taking 
\begin{align}
\label{eq:b_tol}
b_{\textnormal{tol}} = 1.2 \cdot \|\bA \bc^\star - \bu \|_{\cV,2}, 
\end{align}
with $\bc^\star = (\bc^\star_\bmnu)_{\bmnu\in\cJ}$ the coefficients of the Galerkin projection of the solution $u$ of problem \eqref{eq:model_problem} onto $\cV^N$ from \eqref{eq:V_J}.
This choice is sufficient to approximately satisfy the residual constraint of \eqref{eq:SCS_BPDN} up to a small multiplicative constant. In actual applications, one can take $b_{\textnormal{tol}}$ proportional to known {\em a priori} error estimates for the Galerkin projection. 
When $\bA$ and $\bu$ are rescaled with respect to $\hat{\lambda}_{\max}$ as above, we define 
$\hat{\lambda}_{\min} = \lambda_{\min}(\bA\bA^*)$,
and set
\begin{align}
\label{eq:mubar}
\bar{\mu} := \sqrt{\frac{N}{\xi}} \cdot \sqrt{\frac{\hat{\lambda}_{\max}}{\hat{\lambda}_{\min}}}, \qquad \qquad \xi = 10^{-5}.
\end{align}
Our numerical experiments show that this choice of $\bar{\mu}$ is sufficient to yield accurate solutions to the Bregman iteration subproblems 
within a reasonable amount of FPC iterations. 
For example, we find the number of FPC iterations is typically $\mathcal{O}(10^2)$ per Bregman iteration.

\section{Numerical results for the approximation of solutions to the parameterized elliptic PDE with SCS recovery}
\label{sec:Hilbert_valued_numerical_results}

In this section we provide several numerical experiments demonstrating the efficiency of SCS relative to other approaches for solving the parameterized problem \eqref{eq:model_problem}. 
For a general coefficient $a(x,\by)$, we do not know the exact solution $u$ to \eqref{eq:model_problem}. 
Therefore, we check convergence of parametric discretizations against ``highly-enriched'' reference sparse grid stochastic collocation approximations which we denote $u_{h,\textnormal{ex}}(x,\by)$ \cite{Stoyanov:2016ci}. 
More specifically, all approximations (including the enriched reference solution) are computed on fixed finite element meshes $\mathcal{T}_h$ (to be described), and our enriched SC approximation is computed using Clenshaw-Curtis abscissas with level $\ell_{\textnormal{ex}}$, larger than that of the other SC approximations. 
We then approximate the relative errors of the expectation and standard deviation in the $H_0^1(D)$-norm for each of the approximations $u_{h,\textnormal{approx}}$ as
\begin{align}
\label{eq:rel_error_metrics}
\varepsilon^{\textnormal{rel-}\bbE}_{h,\textnormal{approx}} & \approx \frac{\|\bbE_\by[u_{h,\textnormal{ex}} - u_{h,\textnormal{approx}}] \|_{H_0^1(D)}}{\|\bbE_\by[u_{h,\textnormal{ex}}]\|_{H_0^1(D)}} \\
\varepsilon^{\textnormal{rel-}\sigma}_{h,\textnormal{approx}} & \approx \frac{\|\Var_\by[u_{h,\textnormal{ex}} - u_{h,\textnormal{approx}}]^2 \|_{H_0^1(D)}}{\|\Var_\by[u_{h,\textnormal{ex}}]^2\|_{H_0^1(D)}}
\end{align}
where $\bbE_\by[\cdot]$ and $\Var_\by[\cdot]$ denote expectation and variance with respect to $\by\in\cU$, respectively.

In our plots and discussion in this section, we use the following abbreviations. For the PCS and SCS methods, we use: ``PCS-TD'' and ``SCS-TD'' to denote the approximations obtained using the PCS and SCS methods in the total degree subspace with fixed total degree order $p$ to be provided. Similarly, for the stochastic Galerkin (SG) method, we use: ``SG-TD'' to denote the SG approximation obtained in the total degree subspace with increasing order $p$, see, e.g., \cite{Xiu2002,Gunzburger:2014hi}. For the stochastic collocation (SC) method, we use: ``SC-CC'' to denote the sparse grid Smolyak approximation constructed on Clenshaw-Curtis abscissas of level $\ell$ increasing, using the doubling growth rule, see \cite{NTW08}. For the Monte Carlo approximation, we use ``MC.'' 

We follow the convention from \cite{BNTT_comp,Elman2011} in identifying the important {\em stochastic degrees of freedom} (SDOF) as the number of random sample points $m$ for the PCS, MC, and SCS approaches and number of structured sparse grid points $m_\ell$ for the SC method at level $\ell$. 
For the SG method, we define the SDOF to be $N$, the cardinality of the index set used in construction. 
As pointed out in \cite{Dexter2016}, this metric is not an accurate representation of the computational complexity of obtaining the approximation, it provides a useful benchmark in comparing the sample complexity of the sampling-based methods.
We include the SG method only to compare the $L^2_\varrho(\cU)$-optimal (w.r.t. SDOF) error of the Galerkin projection against the error of the sampling-based approximations. 

In all examples excluding MC and SC, we use the orthonormal Legendre polynomials $(L_\bmnu)_{\bmnu\in\cJ}$ in parametric discretization. 
For the MC, PCS, and SCS methods, we use the same set of random sample points $(\by_i)_{i=1}^m$ drawn from the uniform distribution $\varrho$. Since each method relies on random sampling, in order to better quantify the average performance of the algorithms, we use the strategy of running 24 trials on all three methods, fixing the initial seed for the pseudorandom number generator on each trial, and then solving each trial's problem with the same set of $m$ samples. We then average the results over the trials when plotting convergence.
In Section \ref{subsec:comparing_JS_CS} we give a comparison of our SCS approach with PCS, the global recovery through compressed sensing-based polynomial approximation of point-wise functionals. 
Section \ref{subsec:JS_vs_others_linear} compares the SCS approach to the SC, SG, and MC methods described above.

\subsection{Comparison of SCS and CS point-wise functional recovery for global approximation of solutions to parameterized PDEs}
\label{subsec:comparing_JS_CS}

In this section we give a comparison of the SCS and PCS recovery techniques for constructing a fully discrete approximation of the solution $u$ to the parameterized elliptic PDE \eqref{eq:model_problem}. 
Both methods use the same set of random points and corresponding normalized samples $\bu^h = (u_h(x,\by_i)/\sqrt{m})_{i=1}^m\in\cV^m_h$, with $u_h(x,\by_i)$ obtained with the finite element method.
We describe details of obtaining the approximations as follows.
For the SCS method, we solve the unconstrained optimization problem \eqref{eq:SCS_l_V1} with the Bregman-FPC algorithm described in Section \ref{sec:algorithms}, using the measurement matrix $\bA$ given in \eqref{eq:defA} and data $\bu^h$. 
To parameterize the stopping criterion \eqref{eq:Bregman_stop_cond} used in the Bregman algorithm, we set $\bc^\star$ in \eqref{eq:b_tol} equal to the approximation $\bc^{h,\star}$ obtained with finite elements for physical discretization and the SG method for the parametric discretization, in the same total degree subspace used by SCS and PCS.
For the PCS recovery approach, we solve the decoupled standard BPDN problems 
with the same measurement matrix $\bA$ from \eqref{eq:defA} and data $\bm{g}^{(k)} := (u_h(x_k, \by_i))_{i=1}^m\in \R^m$ at all points $x_k$ on the finite element mesh $\mathcal{T}_h$. We apply the standard Bregman-FPC algorithm from \cite{HYZ08,YOGD08}, using the same parameters as the SCS method. 
To construct the fully discrete approximation to $u$ from the point-wise approximations, we interpolate the point-wise GPC expansions in the finite element basis.

Recalling the example from \cite{NTW08}, we study problem \eqref{eq:model_problem} on $D=[0,1]^2$ when parameterized with a deterministic load $f \equiv 1$ and diffusion coefficient $a(x,\by)$ given by 
the affine function 
\begin{align*}
a(x,\by) & = 10 + y_1 \left(\frac{\sqrt{\pi}L}{2}\right)^{1/2} + \sum_{i=2}^d \; \zeta_i \; \vartheta_i(x) \; y_i \numberthis \label{eq:transcendental_a_linear_version} \\
\zeta_n & := (\sqrt{\pi} L)^{1/2} \exp\left( \frac{-\left( \left\lfloor \frac{n}{2} \right\rfloor \pi L\right)^2}{8} \right), \text{ for } n>1, \\
\varphi_n(x) & := \left\{ \begin{array}{rl} \sin\left( \left\lfloor\frac{n}{2}\right\rfloor \pi x_1/L_p \right),  \text{ if $n$ is even,} \\ \cos\left(\left\lfloor\frac{n}{2}\right\rfloor \pi x_1/L_p \right),  \text{ if $n$ is odd.} \end{array} \right.
\end{align*}
Here $y_i \sim U(-\sqrt{3},\sqrt{3})$ (uniform on $(-\sqrt{3},\sqrt{3})$) $\forall i = 1,\ldots, d$, and the constant $10$ is chosen to guarantee $a(x,\by)$ satisfies Assumption \ref{ass:A1}. 
For $x_1\in[0,b]$, let $L_c$ be a desired physical correlation length for the random field $a(x,\by)$, chosen so that the random variables $a(x_1,\by)$ and $a(x_1',\by)$ become essentially uncorrelated for $|x_1-x_1'|\gg L_c$, and define $L_p = \max\{b,2L_c\}$ and $L=L_c/L_p$.

In the following example, with $L_c = 1/4$ and $d=100$ in \eqref{eq:transcendental_a_linear_version}, we construct the SCS-TD and PCS-TD approximations
in a total degree basis of order $p=2$ having cardinality $N={d + p \choose p} = 5151$. 
We compute the random samples of our FE approximations 
using a quasi-uniform mesh $\mathcal{T}_h$ having 206 degrees of freedom, corresponding to a maximum mesh size of $h\approx 1/16$.
The left panel of Figure \ref{fig:CS_vs_CSD} shows a comparison of the relative error statistic $\varepsilon_{h,\textnormal{approx}}^{\textnormal{rel}-\bbE}$ from \eqref{eq:rel_error_metrics} for the SCS-TD and PCS-TD approximations. 
For each data point, we increase the number of random samples $m$ following the rule $m_k = \lceil k N/8 \rceil$ for $k=1,2,\ldots,7$, keeping $N=5151$ fixed.
The top middle and top right panels display the decay of the SG coefficients $\bc^{h,\star}_{\bmnu}$ in $\|\cdot\|_\cV$-norm, before and after sorting by magnitude, and the bottom middle and bottom right panels show the decay of $|\bc_{\bmnu}^{h,\star}(x_i)|$ at a selection of points $x_i$ on the mesh for $\mathcal{T}_h$, before and after sorting by magnitude. 
We note that this choice of parameterization results in a solution $u$ with highly anisotropic coefficient decay, requiring only $\mathcal{O}(100)$ terms to approximate to machine epsilon, and is therefore ideal for sparse approximation methods such as SCS and PCS.

While the difference in errors between the SCS-TD and CS-TD approximations is somewhat dramatic, it should be noted that both solvers are parameterized using the tolerance $b_{\textnormal{tol}} = 1.2 \cdot \|\bA \bc^{h,\star} - \bu\|_{\cV,2}$, a global error statistic, in testing for convergence of the Bregman iterations. 
Since the PCS-TD approximations involve a sequence of point-wise solves, this value of $b_{\textnormal{tol}}$ may lead to poor point-wise reconstructions as it does not account for the error of truncation at a particular point $x_k$ in the mesh. 
In practice however, such point-wise truncation estimates are often unavailable for practical parameterized PDE problems. 
On the other hand, {\em a priori} estimates in global energy norms have been shown for a wide variety of parameterized PDEs under reasonable assumptions on the input data.

\begin{figure}[ht]
\begin{center}
\includegraphics[clip=true,trim=38mm 5mm 40mm 14mm,width=0.98\textwidth]{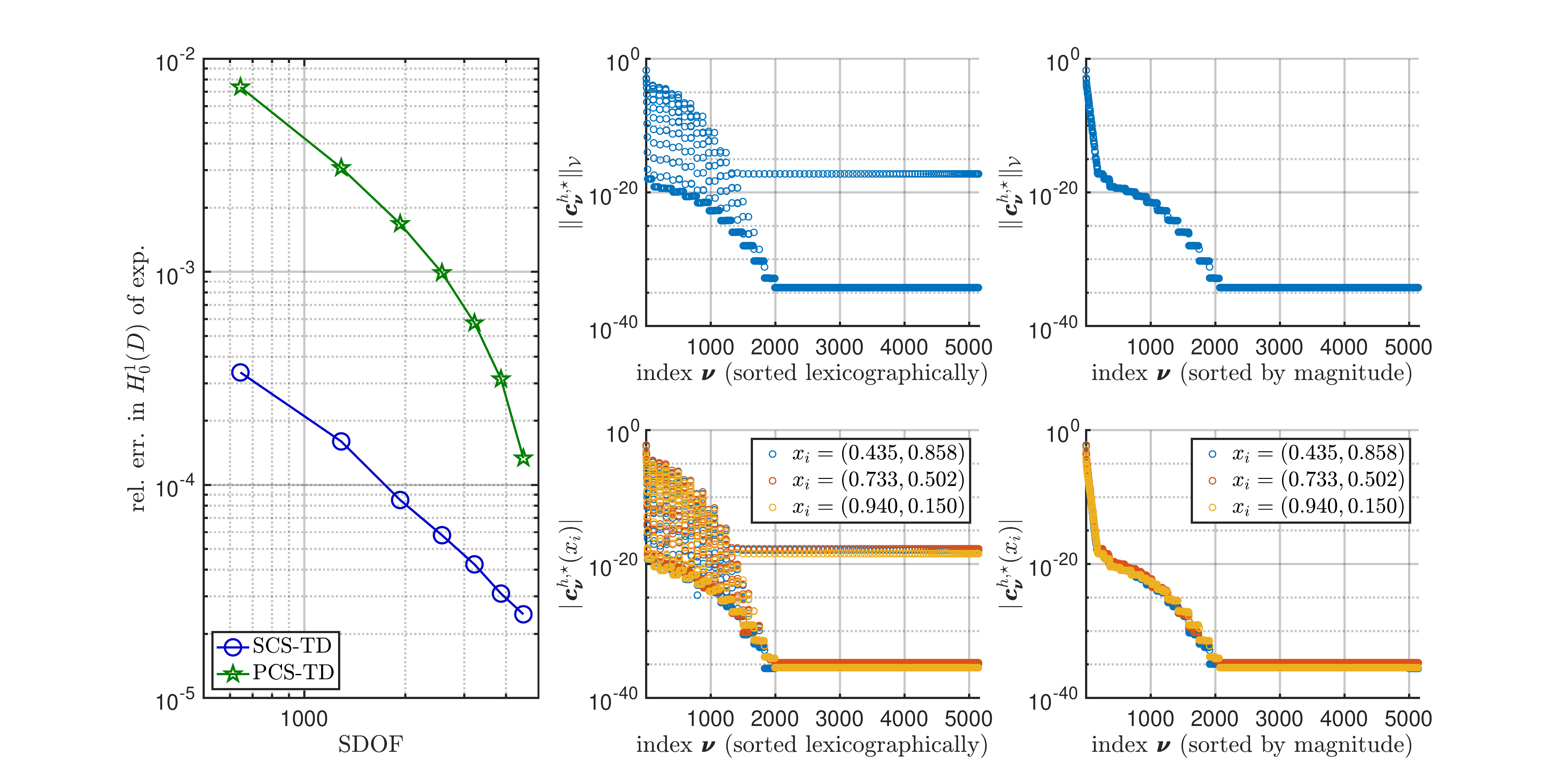}
\end{center}
\caption{\textbf{(left)} Comparison of relative error statistics $\varepsilon_{h,\textnormal{approx}}^{\textnormal{rel}-\bbE}$ from \eqref{eq:rel_error_metrics} for the SCS method (SCS-TD) and compressed sensing point-wise functional recovery (PCS-TD) methods, both computed with a total degree basis of order $p=2$ in $d=100$ dimensions having cardinality $N=5151$. {\bf(center)} Magnitudes of the coefficients in \textbf{(top)} energy norm and \textbf{(bottom)} pointwise at three physical locations $x_i$ sorted lexicographically, {\bf(right)} same as center after sorting by largest in magnitude. Here $\bc^{h,\star}$ is obtained with the stochastic Galerkin method in the same total degree polynomial subspace.}
\label{fig:CS_vs_CSD}
\end{figure}

\subsection{Comparison of SCS with SC, SG, and MC methods for parameterized PDEs}
\label{subsec:JS_vs_others_linear}

In this section, we give a comparison of the approximation errors that can be obtained solving the stochastic elliptic PDE \eqref{eq:model_problem} with the SCS method and the SG, MC, and SC methods. 
We focus on two types of parameterizations, namely the affine coefficient from \eqref{eq:transcendental_a_linear_version} over a range of values of dimension $d$ and correlation length $L_c$, and its log-transformed version, i.e., $\log(a(x,\by) - 0.5)$. 

The first example we study is that of the affine coefficient \eqref{eq:transcendental_a_linear_version} with fixed $L_c=1/4$ and increasing $d$. As in the previous section, the physical FE discretization for this problem uses a fixed quasi-uniform mesh $\mathcal{T}_h$ of $D=[0,1]^2$ having $\Jh=206$ degrees of freedom, i.e., corresponding to a maximum mesh size of $h\approx 1/16$ in $n=2$ physical dimensions. Each row of Figure \ref{fig:CS_vs_others_increasing_dimension} plots the performance of the methods with respect to the relative error metrics $\varepsilon^{\textnormal{rel-}\bbE}_{h,\textnormal{approx}}$ on the left and $\varepsilon^{\textnormal{rel-}\sigma}_{h,\textnormal{approx}}$ on the right from \eqref{eq:rel_error_metrics}, while keeping the order $p$ of the total degree polynomial subspace used in computing the SCS-TD approximation fixed. 
Each row increases $d$, from $20$ on the top row to $60$ on the middle row, and $100$ on the bottom row. 

As mentioned in Section \ref{subsec:comparing_JS_CS}, this choice of correlation length $L_c$ results in highly anisotropic dependence of the solution $u$ on the parameter vector $\by$, with most of the important terms in \eqref{eq:transcendental_a_linear_version} corresponding to low index $i$ of $y_i$. 
Hence, as $d$ increases, the relative sparsity $s/N$ decreases, since $N=\#(\cJ)$ depends exponentially on $d$ for total degree $\cJ$, while the best $s$ terms of \eqref{eq:expansionPC} scale approximately linearly with $d$ under the coefficient \eqref{eq:transcendental_a_linear_version}.
As a result, in higher dimensions the SCS-TD approximation outperforms the SC and MC methods dramatically.
This can be seen as a consequence of the fact that the SC-CC method is using an isotropic growth rule so that $m_\ell$ is growing exponentially in $d$. 
Under an anisotropic growth rule, adapted to the coefficient decay, better performance of the SC-CC approximation would be observed.
However, such anisotropic rules often require detailed knowledge of the parametric dependence. 

Figure \ref{fig:CS_vs_others_increasing_correlation_length} displays the effect of doubling the correlation length $L_c$ in the coefficient $a(x,\by)$ from \eqref{eq:transcendental_a_linear_version} to $1/2$ in $d=100$ dimensions, resulting in an expansion of the solution with even fewer large terms $\|\bc_\bmnu\|_\cV$. 
For the relative standard deviation error statistic $\varepsilon^{\textnormal{rel-}\sigma}_{h,\textnormal{approx}}$, we begin to see the SCS-TD method approaches the error of the final point of the SG-TD method, both of which are constructed using the same basis $(L_\bmnu)_{\bmnu\in\cJ}$ of a total degree polynomial space of order $p=2$.
As $m_k = \lceil k N /8\rceil$, $k=1,2,\ldots,7$, approaches $N = \#(\cJ)$ (the cardinality of the set used in both constructions), the error of the SCS-TD approximation begins to stagnate, suggesting that the Bregman-FPC algorithm has converged to the tolerance $b_{\textnormal{tol}}$. Since $b_{\textnormal{tol}} = 1.2 \cdot \|\bA \bc^{h,\star} - \bu \|_{\cV,2}$, with $\bc^{h,\star}$ the SG solution of total order $p=2$, this suggests the relative errors of both approximations are of the same order. 

Both of Figures \ref{fig:CS_vs_others_increasing_dimension} \& \ref{fig:CS_vs_others_increasing_correlation_length} highlight a key advantage of the SCS approach that is common to all CS-based methods of approximation. Without {\em a priori} knowledge of the coefficient decay, or use of anisotropic weighting in the set $\cJ$, the SCS-TD approximation is able to naturally detect the underlying anisotropy of the problem in refinement as the number of random samples $m$ increases. 
Moreover, our results suggest that simply choosing $\cJ$ large enough to surely cover the best $s$ terms in the expansion \eqref{eq:expansionPC} is sufficient to yield highly accurate approximations. 
These approximations would only be further enhanced if {\em a priori} knowledge of the coefficient decay is given and anisotropic weights are incorporated into the set $\cJ$.
Indeed, we also expect the lower set-based $\ell_1$-weighting strategies for smooth function approximation in our previous work 
\cite{ChkifaDexterTranWebster18} to also yield improved accuracy. We leave a detailed study of such weighting techniques to a future work.
 
Finally, Figure \ref{fig:CS_vs_others_nonlinear_parameterization} shows the result of solving \eqref{eq:model_problem} with the log-transformed example $\log(a(x,\by)-0.5)$ with $a(x,\by)$ from \eqref{eq:transcendental_a_linear_version} and correlation length $L_c = 1/8$, in $d=17$ dimensions. 
For this problem, the stochastic Galerkin system becomes prohibitively dense due to the transcendental dependence on the parameterization. As a result, we were unable to obtain an SG-TD approximation.
Hence, we only compare the SCS, SC, and MC methods, using the SC error as an approximation to $\|\bA\bc^{h,\star}-\bu\|_{\cV,2}$ in parameterizing $b_{\textnormal{tol}}$. 
Here all approximations are computed on a fixed finite element mesh with $\Jh = 713$ degrees of freedom, i.e., corresponding to a maximum mesh size of $h\approx 1/32$. 
We only plot the SCS-TD and MC methods for $m_k = \lceil k N/8 \rceil$ for $k=1,\ldots,4$.
In both the 
$\varepsilon^{\textnormal{rel-}\sigma}_{h,\textnormal{approx}}$ and $\varepsilon^{\textnormal{rel-}\bbE}_{h,\textnormal{approx}}$ error metrics, the SCS-TD method outperforms the MC and SC-CC approximations with respect to the number of samples, even in the case of slower anisotropic decay associated with $L_c=1/8$.
Furthermore, the rate of convergence of the SCS-TD approximation is $\mathcal{O}(m^{-3/2})$, the fastest out of all methods included, further bolstering our claim of the ability of our approach to accurately detect and refine the most important terms of the expansion \eqref{eq:expansionPC}.

\begin{figure}[ht]
\begin{center}
\includegraphics[clip=true,trim=00mm 04mm 14mm 00mm,width=0.40\textwidth]{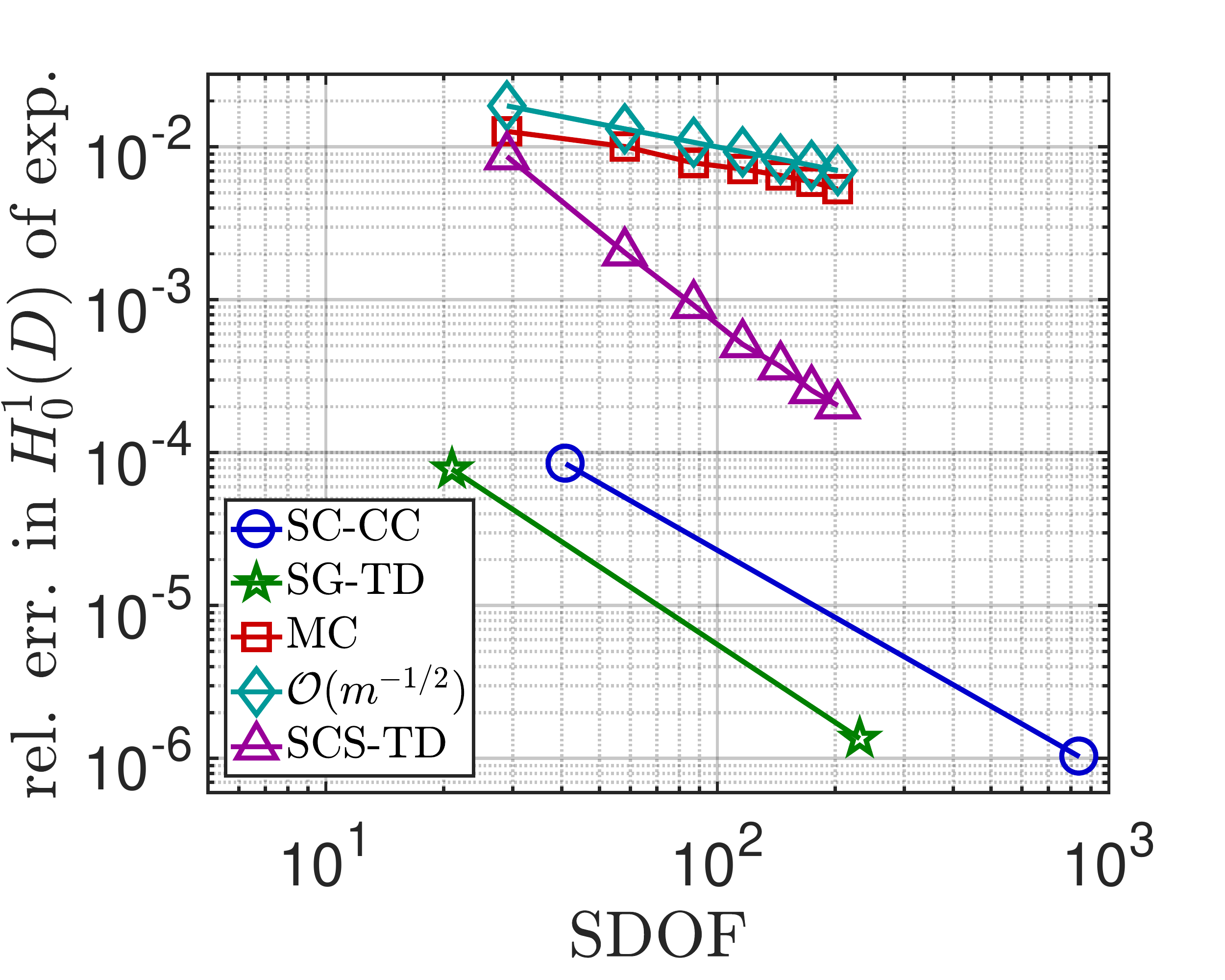}
\includegraphics[clip=true,trim=00mm 04mm 14mm 00mm,width=0.40\textwidth]{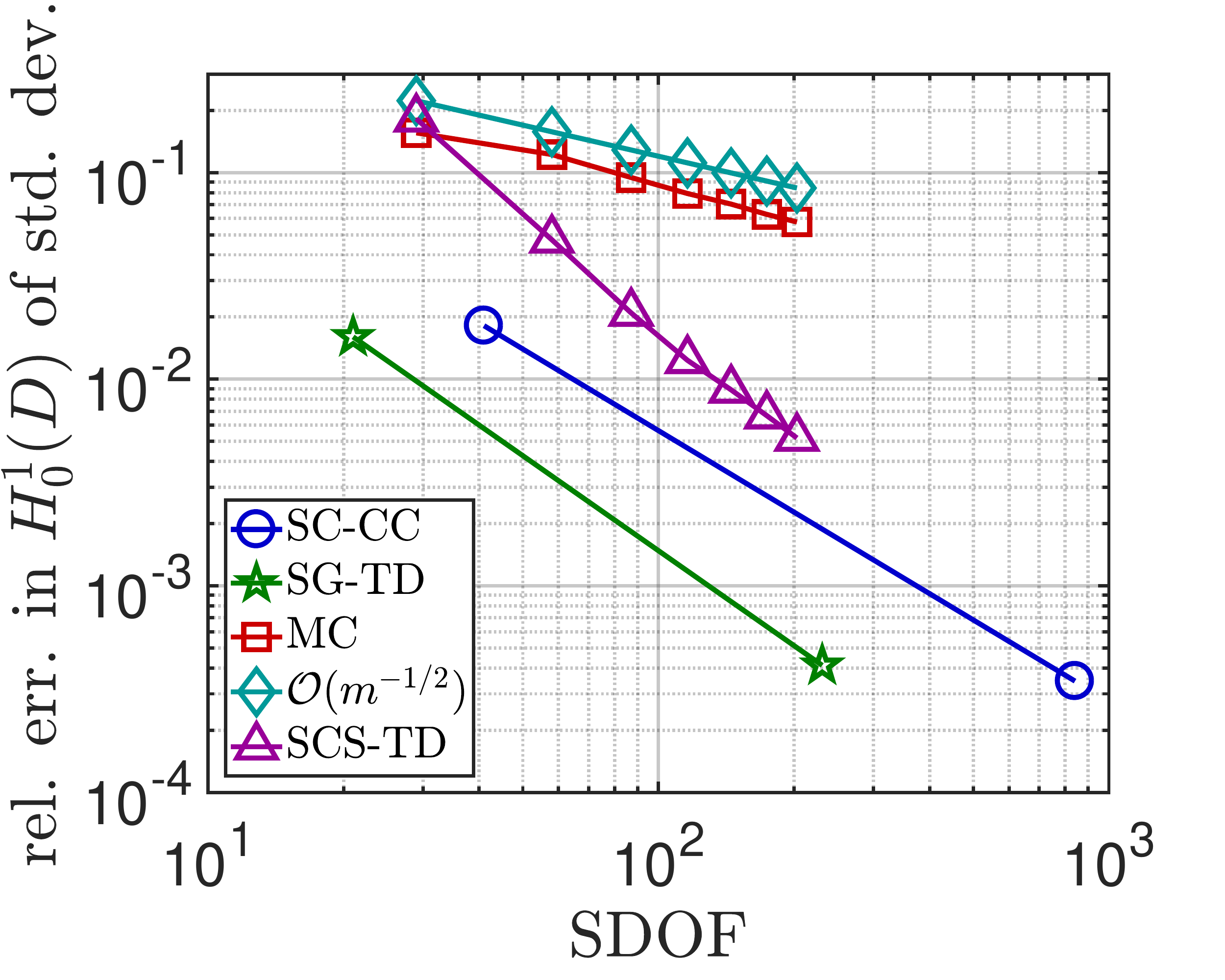}
\includegraphics[clip=true,trim=00mm 04mm 14mm 00mm,width=0.40\textwidth]{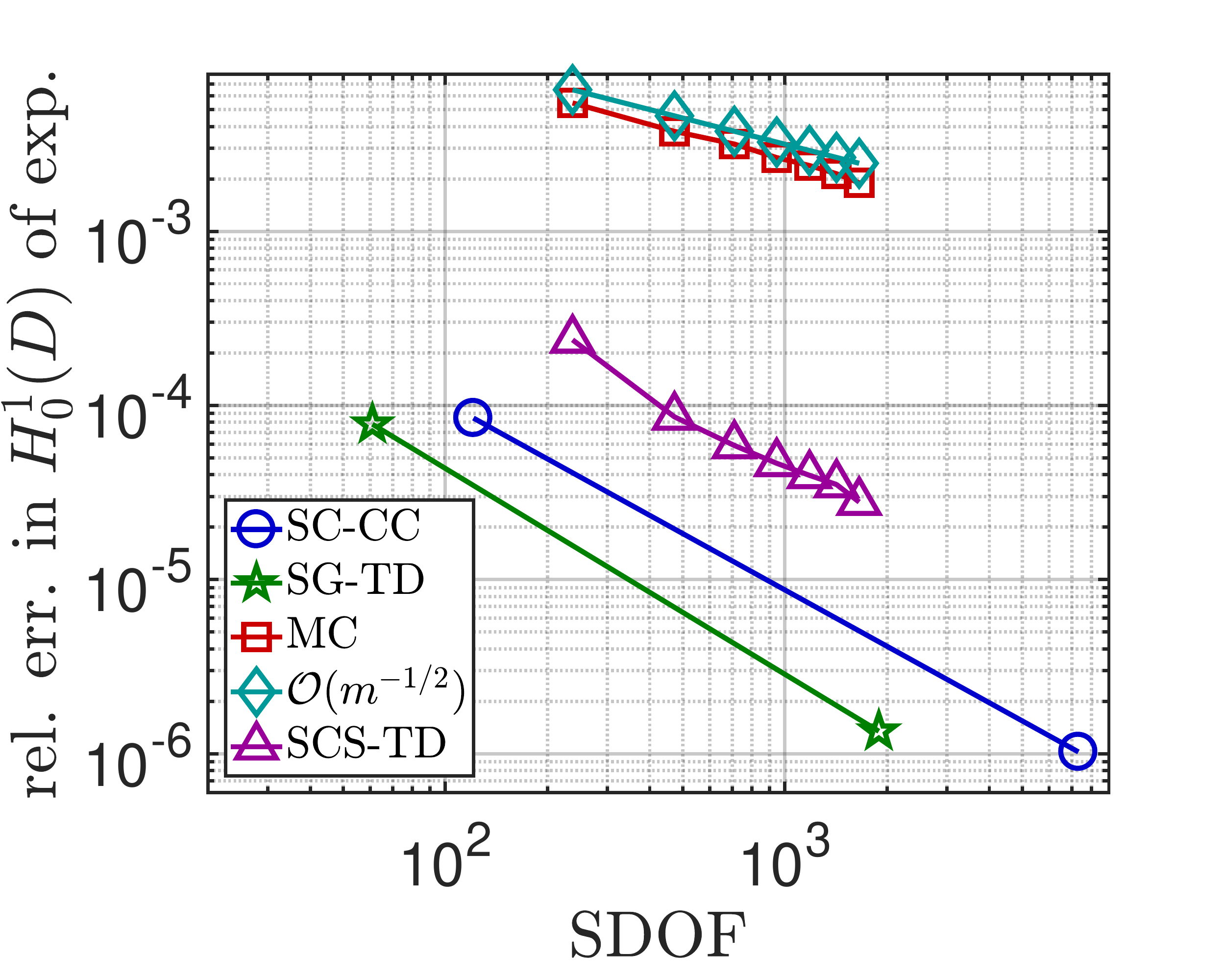}
\includegraphics[clip=true,trim=00mm 04mm 14mm 00mm,width=0.40\textwidth]{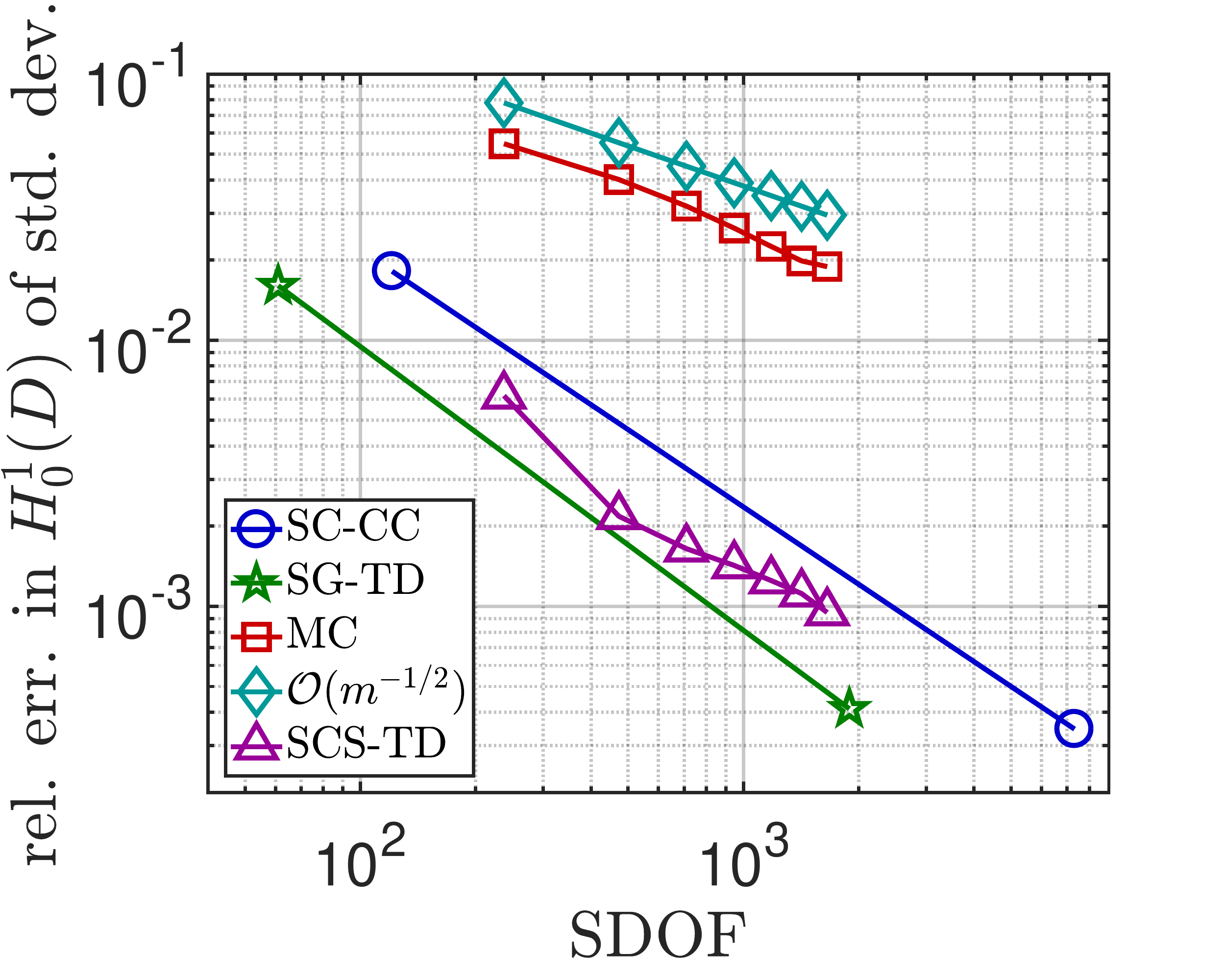}
\includegraphics[clip=true,trim=00mm 04mm 14mm 00mm,width=0.40\textwidth]{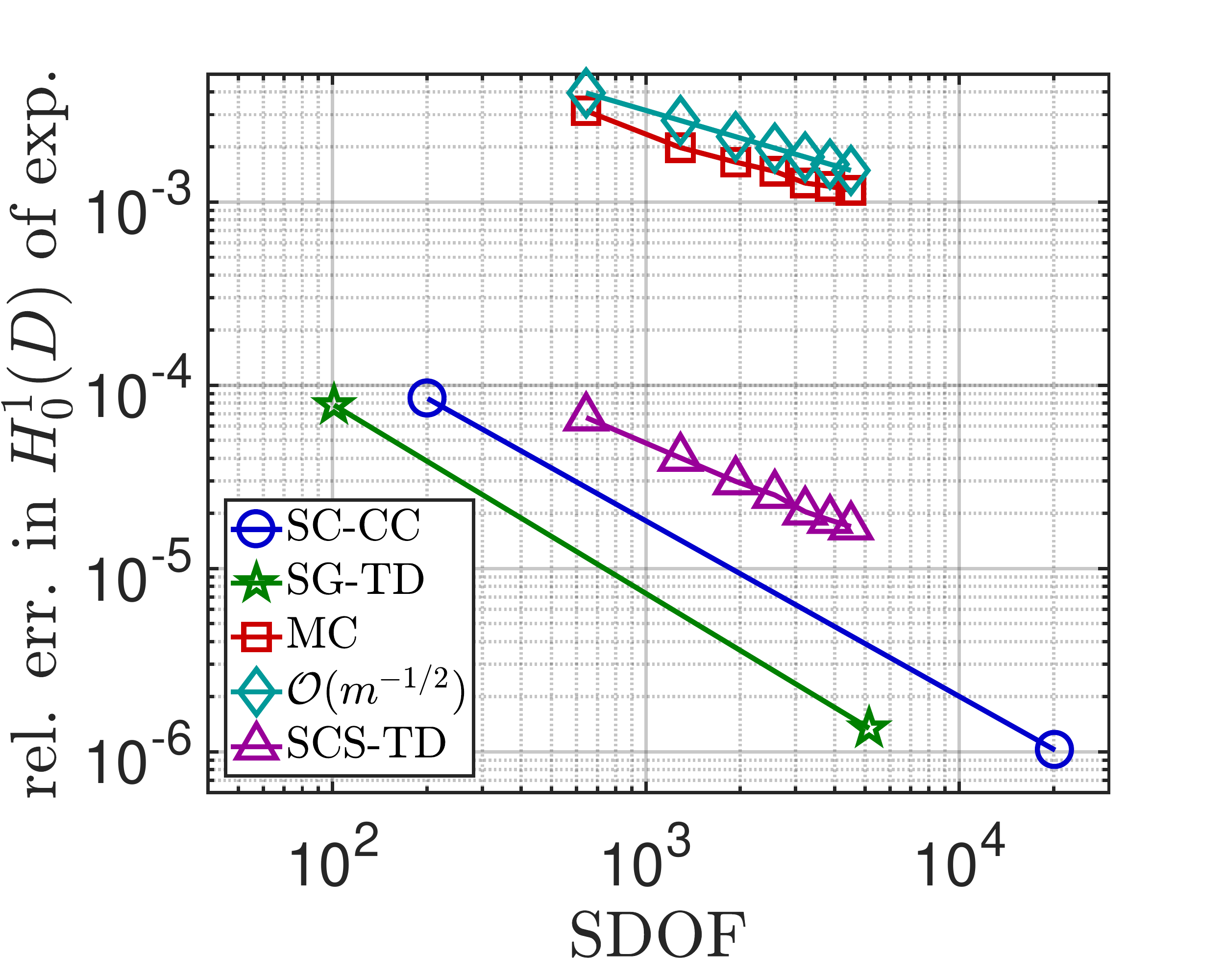}
\includegraphics[clip=true,trim=00mm 04mm 14mm 00mm,width=0.40\textwidth]{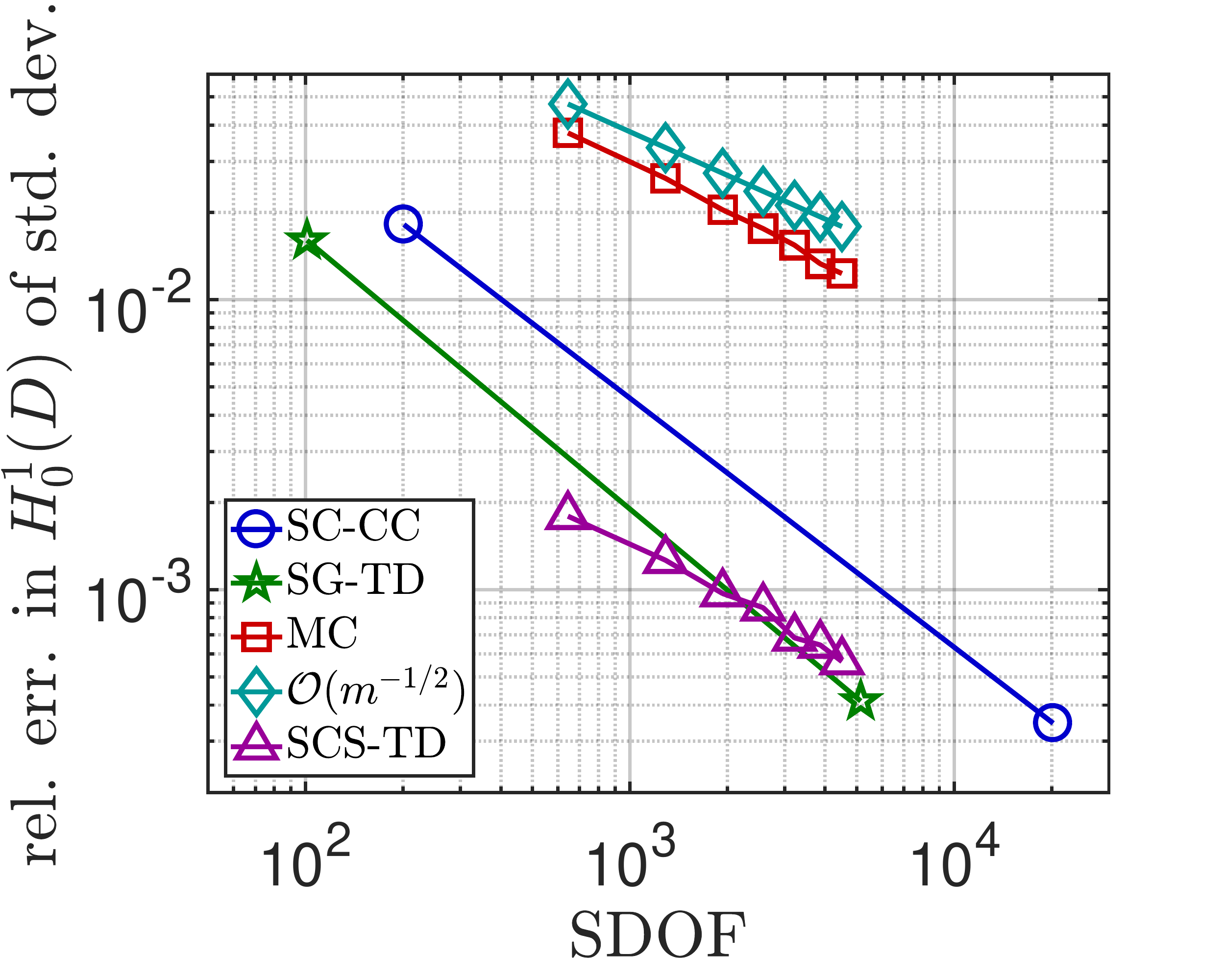}
\end{center}
\caption{Comparison of relative approximation errors $\varepsilon^{\textnormal{rel-}\bbE}_{h,\textnormal{approx}}$ \textbf{(left)} and $\varepsilon^{\textnormal{rel-}\sigma}_{h,\textnormal{approx}}$ \textbf{(right)} from \eqref{eq:rel_error_metrics} for the $L=2,3$ stochastic collocation (SC-CC),  $p=1,2$ stochastic Galerkin (SG-TD), Monte Carlo (MC), and total degree order $p=2$ simultaneous compressed sensing (SCS-TD) methods for solving \eqref{eq:model_problem} with coefficient \eqref{eq:transcendental_a_linear_version} and correlation length $L_c=1/4$. {\bf(top)} $d= 20$, $N=231$, {\bf(middle)} $d=60$, $N=1891$, {\bf(bottom)} $d=100$, $N=5151$.}
\label{fig:CS_vs_others_increasing_dimension}
\end{figure}

\begin{figure}[ht]
\begin{center}
\includegraphics[clip=true,trim=00mm 00mm 10mm 00mm,width=0.40\textwidth]{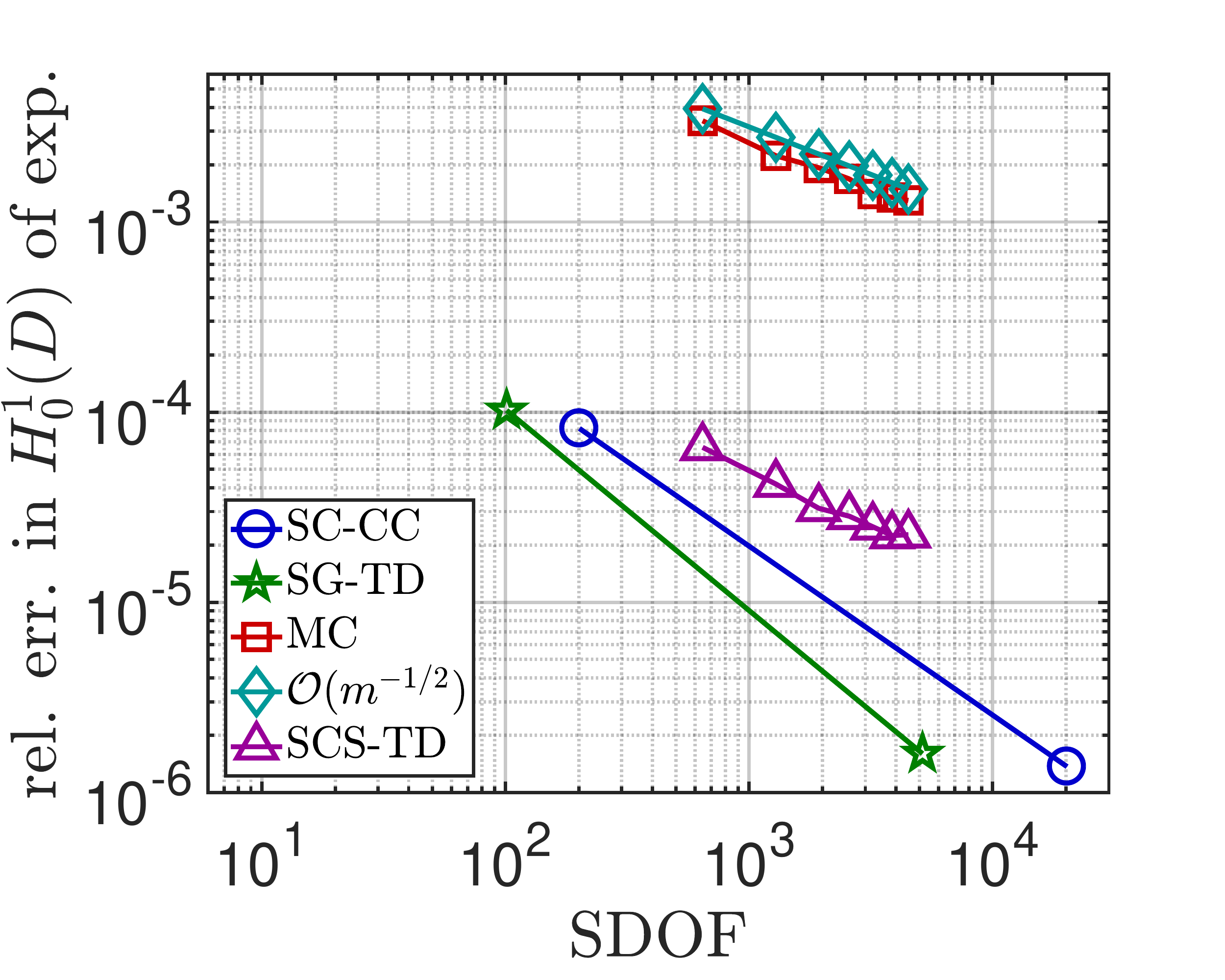}
\includegraphics[clip=true,trim=00mm 00mm 10mm 00mm,width=0.40\textwidth]{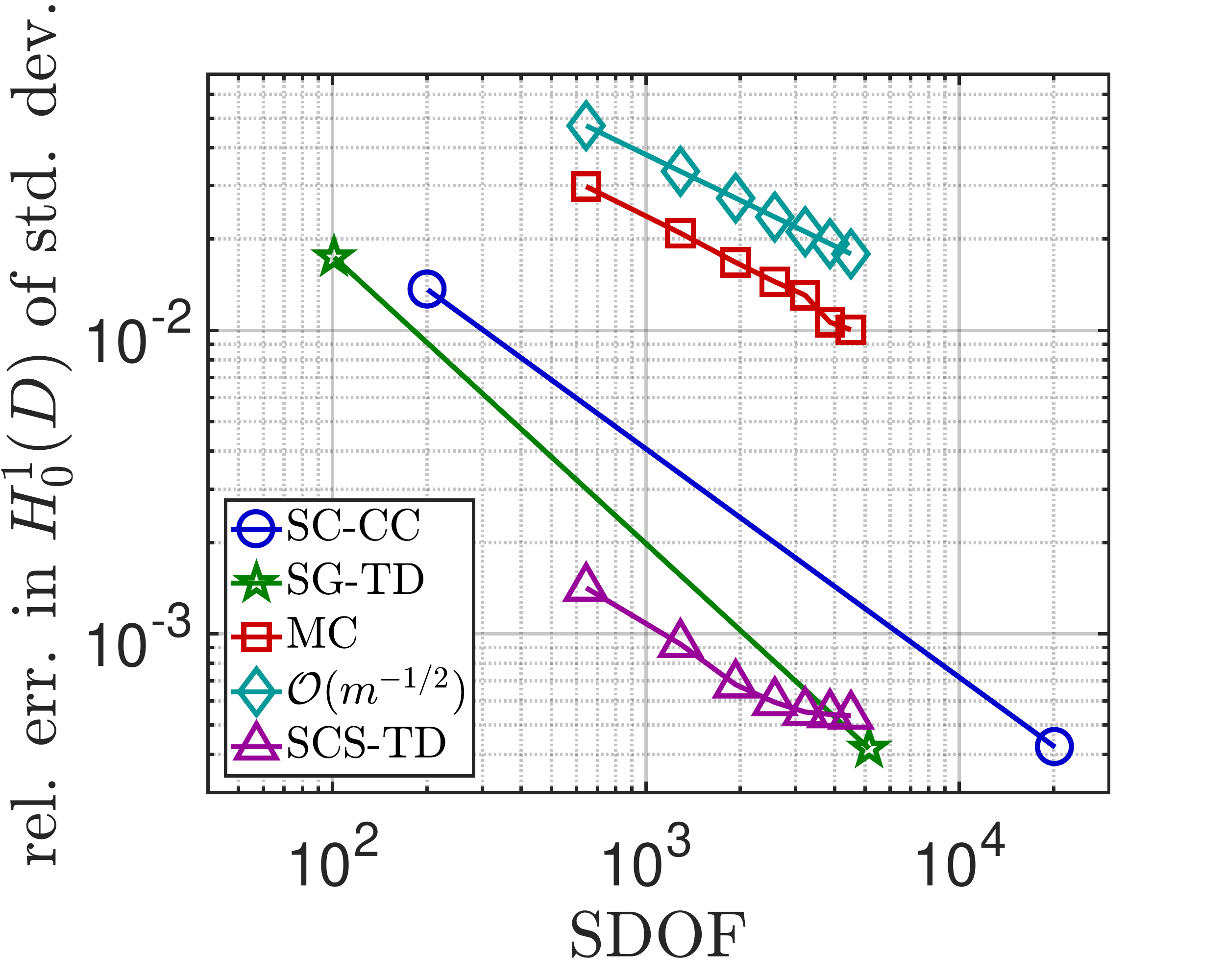}
\end{center}
\caption{Comparison of relative approximation errors $\varepsilon^{\textnormal{rel-}\bbE}_{h,\textnormal{approx}}$ \textbf{(left)} and $\varepsilon^{\textnormal{rel-}\sigma}_{h,\textnormal{approx}}$ \textbf{(right)} from \eqref{eq:rel_error_metrics} for the $L=2,3$ stochastic collocation (SC-CC),  $p=1,2$ stochastic Galerkin (SG-TD), Monte Carlo (MC), and total degree order $p=2$ with $N=5151$ simultaneous compressed sensing (SCS-TD) methods for solving \eqref{eq:model_problem} with coefficient \eqref{eq:transcendental_a_linear_version} and correlation length $L_c=1/2$ in $d= 100$ dimensions.}
\label{fig:CS_vs_others_increasing_correlation_length}
\end{figure}

\begin{figure}[ht]
\begin{center}
\includegraphics[clip=true,trim=00mm 04mm 14mm 00mm,width=0.40\textwidth]{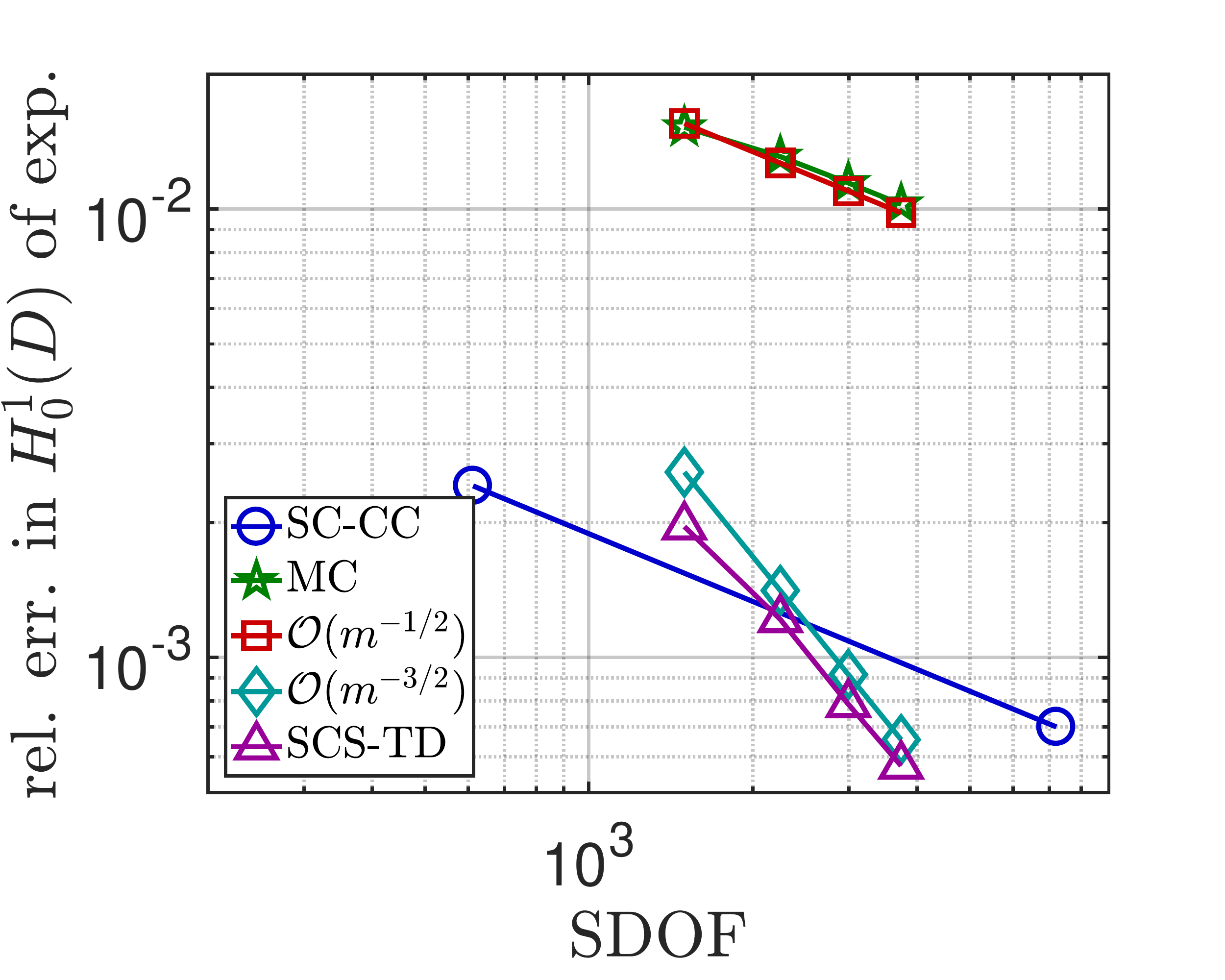}
\includegraphics[clip=true,trim=00mm 04mm 14mm 00mm,width=0.40\textwidth]{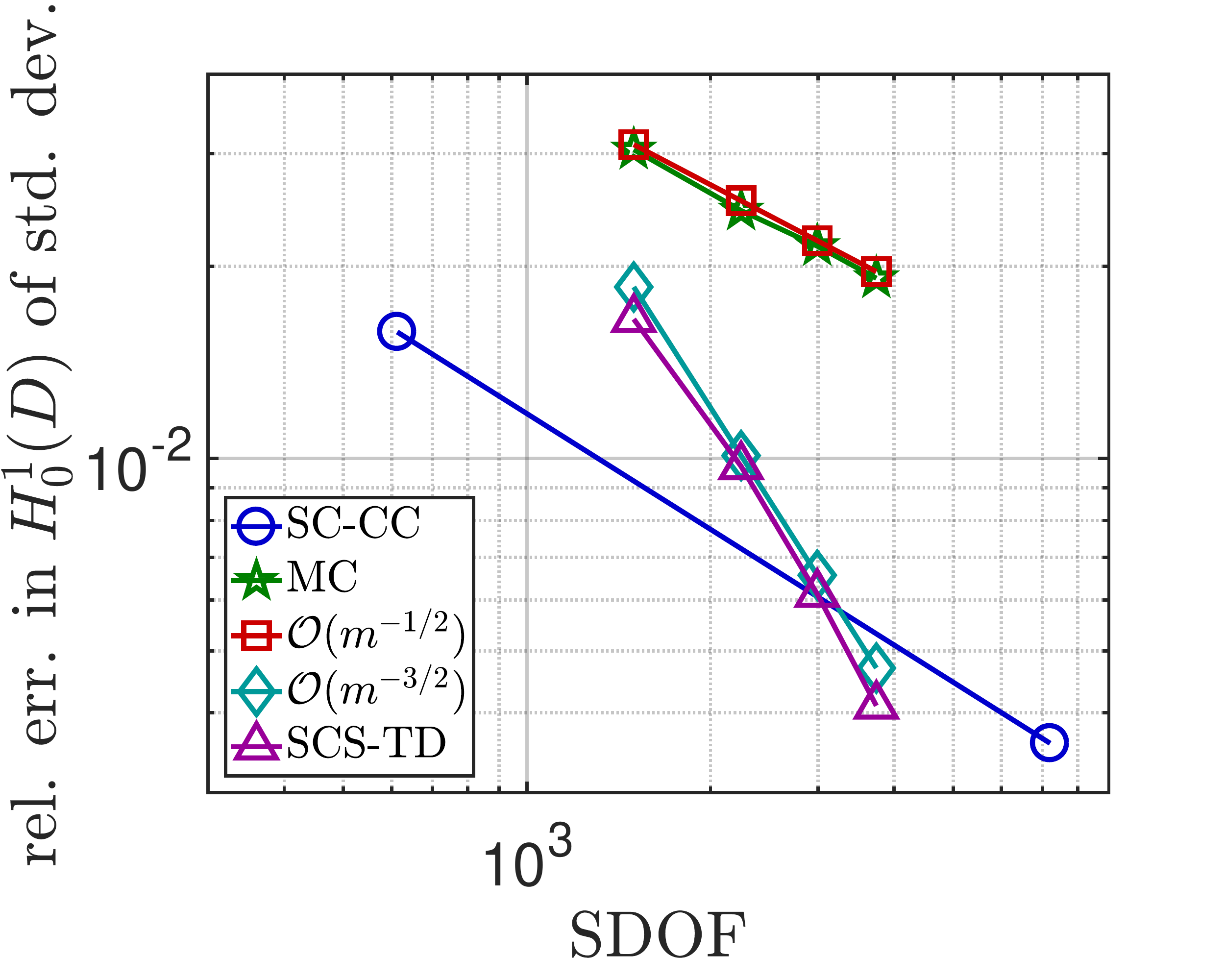}
\end{center}
\caption{Comparison of relative approximation errors $\varepsilon^{\textnormal{rel-}\bbE}_{h,\textnormal{approx}}$ \textbf{(left)} and $\varepsilon^{\textnormal{rel-}\sigma}_{h,\textnormal{approx}}$ \textbf{(right)} from \eqref{eq:rel_error_metrics} for the $L=2,3$ stochastic collocation (SC-CC), Monte Carlo (MC), and total degree order $p=4$ with $N=5985$ simultaneous compressed sensing (SCS-TD) methods for solving \eqref{eq:model_problem} with coefficient $\log(a(x,\by)-0.5)$ with $a(x,\by)$ from \eqref{eq:transcendental_a_linear_version} and correlation length $L_c=1/8$ in $d= 17$ dimensions.}
\label{fig:CS_vs_others_nonlinear_parameterization}
\end{figure}

\section{Conclusion}
\label{sec:conclusion}

In this work, we have presented a novel {\em simultaneous compressed sensing} approach for the approximation of solutions to parameterized PDEs.
In the development of our approach, we derived uniform recovery results for sparse Hilbert-valued vector recovery, combining such results with extensions of error estimates for standard BPDN to the SCS recovery setting and 
quasi-optimal error estimates to prove rigorous bounds on the errors of  SCS approximations.
In particular, we establish fast, sub-exponential rates of convergence under the general assumption that the true solution has orthonormal expansion with rapidly decaying coefficients obeying certain bounds.
Such assumptions have been shown to be valid in context of best $s$-term and quasi-optimal approximations for a large class of parameterized PDE problems.
Moreover, our approach is distinct from previous works on sparse approximations of parameterized PDEs with compressed sensing, as it enables global fully discrete approximation of solutions. 

In the development of our global recovery approach, we relate the SCS reconstruction problem to the concept of joint-sparse recovery. 
Joint-sparse recovery has been studied within the compressed sensing community as a technique for simultaneous recovery of multiple measurement vectors, but thus far had not been applied to parameterized PDEs.
In certain scenarios, one can identify the equivalence of the joint-sparse and Hilbert-valued recovery problems.
However, for general problems we note that the two problems are not equivalent, and choice of norm is key in regularizing problems to yield accurate sparse approximations.

Our numerical results demonstrate the efficiency of our approach with respect to sample complexity, giving comparison to several popular methods of solution.
Indeed, highly accurate approximations can be obtained with a relatively small number of samples in high dimensional problems.
When comparing the approximation errors of our SCS recovery approach with standard Monte Carlo, using exactly the same samples, it is clear that solving the nonlinear interpolation problem via our energy-norm regularized reformulation of standard BPDN yields superior approximations. 
Moreover, our results demonstrate the ability of the SCS recovery method to naturally detect underlying problem anisotropy.
Since our approach is based on a simple modification of standard CS, we expect incorporation of the structured sparse recovery approaches from our previous work \cite{ChkifaDexterTranWebster18} to further enhance the performance of SCS. 
We leave the study of strategies for enhancing smooth function recovery and the overall computational complexity of the SCS recovery problem to a future work.

\bibliographystyle{siam}      
\bibliography{database3}   
\end{document}